\documentclass[a4paper,11pt,openright]{article} 
\usepackage[english]{babel} 
\usepackage[latin1]{inputenc} 
\usepackage{amsmath,ascii,newunicodechar,amssymb,amsthm,amsfonts,amscd,authblk,dsfont,color,cancel,diagbox,enumerate,epsfig,eucal,fancyhdr,latexsym,lmodern,mathrsfs,mathtools,multicol,multirow,musicography,phaistos,skull,simpler-wick,subcaption,tikz-cd}
\usepackage[normalem]{ulem} 
\usepackage[multiuser]{fixme}
\FXRegisterAuthor{nd}{annd}{Nico} 
\FXRegisterAuthor{cvdv}{ancvdv}{Chris} 
\usetikzlibrary{shapes,snakes} 

\usepackage[
final=true,                    
plainpages=false,              
pdfpagelabels=true,            
pdfencoding=auto,              
unicode=true,                  
hypertexnames=true,            
naturalnames=true              
]{hyperref}
\usepackage[all,cmtip]{xy}
\usetikzlibrary{matrix,calc}
\usepackage[autostyle, english = american]{csquotes}
\MakeOuterQuote{"} 

\definecolor{NiColor}{RGB}{77,77,255}
\definecolor{NiColoRed}{RGB}{255,77,77}
\definecolor{NiCitation}{RGB}{0,181,26}

\newtheoremstyle{TheoremStyle}
{3pt}
{3pt}
{\slshape}
{}
{\sc}
{:}
{.5em}
{}

\theoremstyle{TheoremStyle}
\newtheorem{theorem}{Theorem}

\newtheorem{proposition}[theorem]{Proposition}
\newtheorem{lemma}[theorem]{Lemma}
\newtheorem{definition}[theorem]{Definition}
\newtheorem{remark}[theorem]{Remark}
\newtheorem{assumption}{Assumption}

\makeatletter
\def\@endtheorem{\hfill$\lozenge$} 
\makeatother

\title{Classical dynamics of infinite particle systems in an operator algebraic framework}
\author[a]{\href{mailto:t.van_nuland@unsw.edu.au}{T. D. H. van Nuland\footnote{Corresponding author}}}
\author[b]{\href{mailto:chris.ven@fau.de}{Christiaan J. F. van de Ven}}
\affil[a]{Delft University of Technology,  Faculty Electrical Engineering, Mathematics and Computer Science, Building 36, Mekelweg 4
2628 CD, Delft, The Netherlands}
\affil[b]{Friedrich-Alexander-Universit\"at Erlangen-N\"urnberg, Department of Mathematics,
Cauerstra\ss e 11, 91058 Erlangen, Germany}

\def\bearray{\begin{eqnarray}}
\def\earray{\end{eqnarray}}
\def\beq{\begin{equation}}
\def\eeq{\end{equation}}

\def\b0{{\bf 0}}

\def\supp{\operatorname*{supp}}

\def\R{\mathbb R}
\def\Z{\mathbb Z}
\def\mN{\mathcal{N}}
\def\S{\mathcal{S}}

\newsymbol\bt 1202             

\def\N{{\mathbb N}}

\def\bR{{\mathbb R}} 
\def\R{{\mathbb R}}
\def\bS{{\mathbb S}}

\newsymbol\rest 1316         

\def\c{\textnormal{c}}
\def\b{\textnormal{b}}
\def\CR{C_\mathcal{R}}
\def\SR{\mathcal{S}_\mathcal{R}}
\def\DhR{\hat{\mathcal{D}}_{\mathcal{R}}}
\newcommand{\supnorm}[1]{\|#1\|_\infty}
\newcommand{\spn}{\operatorname{span}}
\newcommand{\ran}{\operatorname{ran}}
\newcommand{\pos}{\textnormal{pos}}
\newcommand{\mom}{\textnormal{mom}}

\begin{document}

\maketitle
\begin{abstract}
We construct C*-dynamical systems for the dynamics of classical infinite particle systems describing harmonic oscillators interacting with arbitrarily many neighbors on lattices, as well on more general structures.
Our approach allows particles with varying masses, varying frequencies, irregularly placed lattice sites and varying interactions subject to a simple summability constraint.  A key role is played by the commutative resolvent algebra, which is a C*-algebra of bounded continuous functions on an infinite dimensional vector space, and in a strong sense the classical limit of the Buchholz--Grundling resolvent algebra, which suggests that quantum analogs of our results are likely to exist. 
For a  general class of Hamiltonians, we show that the commutative resolvent algebra is time-stable, and admits a time-stable sub-algebra on which the dynamics is strongly continuous, therefore obtaining a C*-dynamical system.
\end{abstract}

\tableofcontents

\section{Introduction}\label{Introduction}
In this work we study the time evolution of classical infinite particle systems. We let $\Gamma$ be a countable set of sites -- for instance, a lattice in $\R^\ell$. Around each site a particle is pinned by a harmonic force, and each pair of particles interact through attractive and/or repulsive forces. If $q_k$ denotes the displacement of a particle from its chosen lattice position $k\in\Gamma$, and $p_k$ denotes its momentum, our Hamiltonian can be formally written as
\begin{align}\label{eq:H intro}
H(p,q)=\sum_{k\in\Gamma}\left(\frac{\|p_k\|^2}{2m_k}+\frac{\nu_k\|q_k\|^2}{2}\right)+\sum_{k,l\in\Gamma}V_{kl}(q_k-q_l),
\end{align}
for varying constants $m_k,\nu_k>0$ and varying bounded interaction potentials $V_{kl}$ subject to certain smoothness and summability constraints; precise definitions are in Section \ref{Sec:oscillating lattice systems}.
In contrast to many other works in this context, we do not assume that the masses $m_k$ and force constants $\nu_k$ of the particles are identical, nor do we place any constrains on the nature of the lattice, phrasing the thermodynamic limit purely in terms of the partial order of finite subsets of $\Gamma$ with inclusion. As a consequence, our results hold for general solid-state material structures $\Gamma\subseteq\R^\ell$ that are not necessarily arranged in a periodic lattice. Practical examples range from crystals to glass, from doped metals to carbon nanotubes. From the extremely chaotic and non-integrable dynamical behavior of such a system, we nonetheless extract certain ``asymptotically linear'' characteristics, which we use to fully translate the phase space description of the system into an algebraic description, resulting in a C*-dynamical system.
This provides a general algebraic framework in which fundamental questions about thermodynamic equilibrium of many particle systems can be raised and, hopefully, be answered in future projects \cite{Aizenman_Goldstein_Gruber_Lebowitz_Martin_1977,Drago_vandeVen_2023,Gallavotti_Verboven_1975,vandeVen_2022_2}. 

The commutative resolvent algebra is a commutative C*-algebra introduced in \cite{vanNuland_2019}, and shown there to be the classical limit of the Buchholz--Grundling resolvent algebra $\mathcal{R}(X)$ \cite{Buchholz_Grundling_2008} in a strict sense.
The latter algebra has proved to be an excellent framework for the study of  Bose-Einstein condensation \cite{Detlev_Buchholz_2022,Detlev_Buchholz_2022_2,Buchholz_Bahns_2021}, and is known to model quantum lattice systems with an infinite number of degrees of freedom in a large class of settings \cite{Buchholz_Grundling_2015,Detlev_Buchholz_2017}, although not yet for systems as general as considered here in the classical case. 
  Classical models are often more intuitive, as their observable algebras are simply given by function algebras defined on phase space. We use this phase space intuition to vastly generalize the dynamics known to work, and thus create a blueprint \cite{Galitski, Vuillermot_1980} for what can be hoped for in the quantum case.

To motivate the need for the (commutative) resolvent algebra let us consider the case of spin lattice systems for which the underlying configuration space is typically taken to be an infinite product of spheres $\bS^d$, which is compact as a result of Tychonoff's theorem. In this setting, there is no ambiguity about the underlying algebraic structure encoding the observables, since $C_\c(\Pi_{\mathbb{Z}^\ell} \bS^d)=C_0(\Pi_{\mathbb{Z}^\ell}\bS^d)=C(\Pi_{\mathbb{Z}^\ell}\bS^d)$, where $C_\c$ denotes the compactly supported continuous functions, $C_0$ the continuous functions vanishing at infinity and $C$ all the continuous functions.  This drastically changes when the single particle phase space is not compact anymore. Indeed, the usual C*-algebra $C_0(X)$ is trivial for infinite dimensional normed vector spaces $X$, since compact subsets of $X$ have empty interior by Riesz' Lemma.
In this more general context various rigorous results are often known in a measure theoretical, rather than an algebraic setting \cite{Georgii_1988,Ruelle_1969,PPT_1976}. Despite the fact that these approaches have shown their importance in classical statistical mechanics, one may still hope for a fully C*-algebraic description of interacting dynamical lattice systems: this would allow for an abstract
description where the features of physical observables are encoded by algebraic relations which are well manageable in C*-algebras. This is precisely where the commutative resolvent algebra, denoted by $\CR(X)$, comes to the rescue. 

In contrast to $C_0(X)$, the C*-algebra $\CR(X)$ is  a non-trivial subalgebra of $C_b(X)$, and yet, for finite subspaces $V\subseteq X$ the algebra $C_0(V)$ is embedded in $\CR(X)$ by the pull-back of the projection onto $V$.
In fact, $\CR(X)$ is the inductive limit of such $C_0(V)$, and thus forms a conveniently small class of well-behaved quasi-local classical observables (cf. \cite{Buchholz_Grundling_2008,vanNuland_2019,vanNuland_2022}), as opposed to $C_b(X)$.
This inductive limit furthermore makes it possible to study 
the thermodynamic limit. A major advantage is that  physical emergent features such as the occurrence of spontaneous symmetry breaking and  phase transitions in infinite systems can  be analyzed from  approximations of the relatively well-known finite subsystems \cite{Moretti_vandeVen_2021,vandeVen_2022}, which describe these phenomena in real matter \cite{vandeVen_2023}.

The first main contribution of this paper, Theorem \ref{thm: CR stable}, is to show that the commutative resolvent algebra
is invariant under the time evolution of any Hamiltonian $H$ of the very general form \eqref{eq:H intro}. 
The second main contribution, Theorem \ref{thm: final goal}, then shows that, for each specific $H$ (fixing the numbers $\{m_k\}_{k\in\Gamma},\{\nu_k\}_{k\in\Gamma}$ and potentials $\{V_{kl}\}_{k,l\in\Gamma}$) the C*-algebra $\CR(X)$ has a nontrivial subalgebra that forms a C*-dynamical system for this specific dynamics. The latter actually follows quite easily from the former by an adapted argument of \cite{Buchholz_Grundling_2008}.
Different dynamics may induce different subalgebras, but they will always be contained in $\CR(X)$, therefore making them accessible for subsequent analysis concerning states, symmetry breaking, phase transitions, et cetera.
\\\\
Compared to traditional measure-theoretic and probalistic approaches, the algebraic approach has several advantages, of which we highlight two.

A first advantage is that the algebraic approach enables the study of classical KMS states. The classical KMS condition, introduced in \cite{Gallavotti_Verboven_1975}, is an analogue of the usual (quantum) KMS condition, now formulated in terms of a commutative 
C*-algebra, a Poisson bracket, and a globally defined automorphism, all of which are present in our current setting. Solutions to this equation, known as classical KMS states, are expected to encode the equilibrium properties of the system as a function of temperature \cite{Drago_Pettinari_Ven,Drago_Pettinari_Ven2,Drago_vandeVen_2023}. In contrast to the quantum case, at the classical level there exists another much better known notion of thermodynamic equilibrium, namely that described by the Dobrushin--Lanford--Ruelle (DLR) equations and Gibbs measures. In general, however, these two notions do not necessarily coincide. The commutative resolvent algebra therefore opens up an entirely new domain of the study of thermodynamic equilibrium through both these concepts.

A second advantage is the possibility of studying incomplete flows via partial actions. Since we work with automorphisms of a  C*-algebra rather than with flows on a classical phase space, we obtain a dynamical framework in which the (non-trivial) \emph{ideal structure} of the commutative resolvent algebra plays 
a fundamental role \cite{Exel}. This perspective becomes particularly powerful in situations where the 
Hamiltonian flow fails to be globally defined --as in systems with Coulomb  interactions. In such cases, the theory of \emph{partial actions}  \cite{Exel} provides the appropriate framework: it allows the dynamics to be restricted to those closed ideals on which the automorphisms are well defined,  thereby yielding partial dynamical systems. As a result, the study of the ideal structure is essential for understanding the dynamics. 
\\\\
From a technical point of view our proof profits from the strategy applied in \cite{vanNuland_Stienstra_2019} where phase space analysis is used to prove stability for finite systems of particles with a compact configuration space. To extend these results to noncompact configuration spaces is however still a technical exercise. 

The subsequent extension to infinite systems is done by an elegant Dyson-series argument, inspired by and extending \cite{Robinson}. Our extension neatly shows how quantum reasoning can be carried over to classical reasoning when the right smoothness criteria are employed, why one can still far exceed nearest-neighbor interactions, and why one can remove all dependence on the metric structure of $\Gamma$.

The intent of making the thermodynamic limit completely independent of a metric is inspired by the ideas of noncommutative geometry and the (dual) continuum limit in lattice gauge theory. Both topics are often studied with a C*-algebraic description of physics in mind, and both are only yet rigorous in the classical case.
Our general set-up seems to be suitable for the inclusion of such more abstract models, as well as the inclusion of curved configuration spaces and possibly Coulomb interactions.
\\\\
The paper is organized as follows: Section  \ref{Sec:oscillating lattice systems} is devoted to a brief introduction to interacting and oscillating particle systems and the commutative resolvent algebra. In Section \ref{classical dynamical system resolvent} the Hamiltonian dynamics is introduced for which we show in Section \ref{Sect:finite cases} that it leaves invariant the commutative resolvent algebra in the case of finite systems. Secondly, by use of a summability condition on the interactions, this result is extended to infinite systems by means of Theorem \ref{thm: CR stable} in Section \ref{Section: infinite dynamics}. In Section \ref{Sec:strong continuity}, strong continuity of the time evolution on certain subalgebras yields our final result, Theorem \ref{thm: final goal}. We discuss our findings and future prospects in Section \ref{sct:Discussion}.




\section{Preliminaries}\label{Sec:oscillating lattice systems}
We introduce our conventions and assumptions regarding the classical oscillating and interacting (in)finite particle systems under consideration. 

\subsection{The index set and its thermodynamic limit}
We consider an arbitrary countable set $\Gamma$ -- typically interpreted as a discrete subset of $\R^\ell$, namely, as the set of points of confinement around which the particles are pinned by a harmonic potential. 
This already endows the set of finite subsets of $\Gamma$ with a partial order (inclusion), which is upward directed and hence defines a notion of convergence -- the one of nets.\footnote{This notion of convergence is the strongest possible one on subsets of $\Gamma$ in which every point of $\Gamma$ is eventually absorbed in the subsets, and implies for instance convergence with respect to subsets of hypercubes, et cetera.} Explicitly, if $F$ is any function defined on finite subsets of $\Gamma$ and taking values in a normed space,
\[
\lim_{\Lambda \nearrow \Gamma} F(\Lambda) = \alpha
\]
means that for every $\varepsilon > 0$, there exists a (finite) subset $K_\epsilon \Subset \Gamma$ such that for all $K_\epsilon\subseteq \Lambda\Subset\Gamma$ we have
\[
\|F(\Lambda) - \alpha\| < \varepsilon. 
\]

Assumptions on the material topology and geometry will be encoded not in $\Gamma$ but in our assumptions on the interaction potentials, to be discussed in \NAK \ref{subsection: Hamiltonian}. 

\subsection{The phase space}\label{sct:subsection The phase space}
To each element of $\Gamma$ we associate a phase space $\mathbb R^{2d}$, and our total phase space is given by
$$\Omega=\ell_\textnormal{c}(\Gamma,\R^{2d})=\ell_\textnormal{c}(\Gamma,\R^d)\times \ell_\textnormal{c}(\Gamma,\R^d)$$
consisting of pairs $\omega=(p,q)\in\Omega$ of finite sequences $p=(p_l)_{l\in\Gamma},q=(q_l)_{l\in\Gamma}$ for which each entry takes values in $\R^d$.
The components of $p_l$ (resp. $q_l$) in $\R^d$ are denoted $p_{l,i}$ (resp. $q_{l,i}$) for $i=1,\ldots,d$. 

Let $\Lambda\Subset\Gamma$ be any finite subset labeling the particles of a subsystem.
For $p\in\Omega_\Lambda^\mom\cong\R^{|\Lambda|}\otimes\R^d$, the precise definition of $\Omega_\Lambda^\mom$ given below,
we sometimes use the identification $p=\sum_{k\in\Lambda}e_k\otimes p_k\in\R^{|\Lambda|}\otimes \R^d$,
with $e_k$ the $k^{th}$ standard basis vector of $\R^{|\Lambda|}$, i.e. $(e_k\otimes p_k)_l=\delta_{k,l}p_k\in\R^d$, and similarly for $q$.

By construction $\Omega$ is a countably infinite dimensional vector space which admits a natural inner product induced by the inclusion $\Omega\subseteq \ell^2(\Gamma,\R^{2d})$ into the square-summable sequences.
For interpretational purposes, one may rather regard $\ell^2(\Gamma,\R^{2d})$ as the underlying phase space, and Remark \ref{rmk: construction infinite system} shows that this is valid. In that sense, $\Omega$ is an auxiliary dense subset of the phase space, used to define the observable algebra, but for brevity we shall refer to $\Omega$ as the phase space.

We moreover define, for each finite subset $\Lambda\Subset\Gamma$,
$$\Omega_\Lambda:=\{(p,q)\in\Omega:~p_l=q_l=0\text{ for }l\notin\Lambda\}\cong \R^{2|\Lambda|d}.$$
We occasionally adopt the notation
\begin{align*}
    \Omega_\Lambda^\mom=&\{(p,0)\in\Omega_\Lambda\}\cong \R^{|\Lambda|d};\\ 
\Omega_\Lambda^\pos=&\{(0,q)\in\Omega_\Lambda\}\cong \R^{|\Lambda|d},
\end{align*}
and we note that $\Omega_\Lambda=\Omega_\Lambda^\mom\oplus\Omega_\Lambda^\pos$.
We furthermore emphasize that
$$\Omega=\bigcup_{\Lambda\Subset\Gamma}\Omega_\Lambda.$$

\subsection{The commutative resolvent algebra}\label{Commutative resolvent algebra}
The first step in constructing an operator algebraic framework for a physical system is the construction of a C*-algebra defining the class of relevant observables. 
The observable algebra that this paper proposes is called the commutative resolvent algebra \cite{vanNuland_2019} of $\Omega$, defined as follows.

Let $X$ be a (possibly infinite-dimensional) real-linear inner product space. The {\em commutative resolvent algebra} of $X$, denoted $\CR(X)$, is the C*-subalgebra of the algebra of bounded operators $C_\b(X)$ generated by resolvent functions on $X$, i.e.,  functions of the form
\begin{align*}
	h_x^\lambda(y):=1/(i\lambda- x\cdot y),
\end{align*}
for $x\in X$, $\lambda\in\mathbb{R}\setminus\{0\}$.\footnote{Actually $\CR(X)$ is already generated as a Banach space by $h^\lambda_x$, as shown in \cite[Theorem 4.4 and (2)]{vN23}.} 
The inner product on $X$ yields a norm $||\cdot||$ and a topology with respect to which $h_x^\lambda:X\to\mathbb C$ is a continuous function.

We consider $\CR(\Omega)$, with $\Omega=\ell_c(\Gamma,\R^{2d})\subseteq\ell^2(\Gamma,\R^{2d})$ of Section \ref{sct:subsection The phase space}.
Conveniently, $\CR(\Omega)$ is the inductive limit of the net of all $\CR(V)$, where $V\subset \Omega$  ranges over all finite-dimensional subspaces of $\Omega$, and the connecting maps defining this limit are the pull-backs of the projection maps $W\twoheadrightarrow V$ for $V\subset W$. This remains true if one restricts the net to any cofinal class of finite-dimensional subspaces. It follows that 
$$\CR(\Omega)=\varinjlim \CR(\Omega_\Lambda),$$ where the inductive limit is taken with respect to the pull-backs of ${\pi_\Lambda}|_{\Lambda'}:\Omega_{\Lambda'}\to\Omega_{\Lambda}$ $~(\Lambda\subseteq\Lambda'\Subset\Gamma)$, where $\pi_\Lambda:\Omega\to\Omega_{\Lambda}$ is the orthogonal projection onto $\Omega_\Lambda$. More explicitly, we have (cf. \cite[Proposition 6.2.4(i)]{inductive})
\begin{align}\label{eq:CR closure of union}
\CR(\Omega)=\overline{\bigcup_{\Lambda\Subset\Gamma}(\CR(\Omega_\Lambda)\circ\pi_\Lambda)}.
\end{align}
Hence, $\CR(\Omega)$ is an algebra of ``quasi-local'' observables \cite{Landsman}, containing a dense subalgebra of ``local'' observables, i.e. functions that only depend on finitely many particles.

\begin{remark}\label{rmk: construction infinite system}
      One can consider $\CR(\Omega)$ as a space of functions on the larger phase space $Y=\ell^2(\Gamma,\R^{2d})$. Indeed, the canonical projections $Y\twoheadrightarrow\Omega_\Lambda$ induce injective *-homomorphisms $\CR(\Omega_\Lambda)\hookrightarrow C_\b(Y)$ which by inductive limit induce an injective *-homomorphism $\CR(\Omega)\hookrightarrow C_\b(Y)$.
      As such, $\CR(\Omega)$ is also an observable algebra on the phase space $Y$, and in fact, distinguishes all points of $Y$.
\end{remark}

\subsection{A dense Poisson algebra}
We  define the subspace $\SR(\Omega)\subset \CR(\Omega)$ as the span of so-called levees $g\circ \pi$ for which $g$ is Schwartz, namely
\begin{align}\label{eq: levees}
	\SR(\Omega):=\text{span}\{ g\circ \pi ~|~\pi\text{ fin. dim. projection on $\Omega$, }\ g\in \S(\ran(\pi)) \},
\end{align}
where $\S(\text{ran}(\pi))$ denotes the Schwartz space on $\text{ran}(\pi)$. More generally, a ``levee'' is a function $f = g \circ \pi\in \CR(\Omega)$ for a finite dimensional projection $\pi$ and a function $g \in C_0(\text{ran}(\pi))$. By \cite[Prop. 2.4]{vanNuland_2019} the set $\SR(\Omega)$ is a dense *-subalgebra of $\CR(\Omega)$.

We can put a Poisson bracket on $\SR(\Omega)$ by use of the canonical Poisson bracket on $C^\infty(\Omega_\Lambda)\cong C^\infty(\R^{2|\Lambda|d})$, as follows. 
For any two functions $f_1,f_2\in \SR(\Omega)$ we can choose $\Lambda\Subset\Gamma$ large enough such that $f_1=g_1\circ \pi_{\Lambda}$ and $f_2=g_2\circ \pi_\Lambda$ for functions $g_1,g_2\in \SR(\Omega_\Lambda)\subseteq C^\infty(\Omega_\Lambda)$, and where $\pi_\Lambda:\Omega\to\Omega_\Lambda$ denotes the orthogonal projection.
We define
\begin{align*}
    \{g_1\circ \pi_{\Lambda},g_2\circ \pi_{\Lambda}\}:=\{g_1,g_2\}_\Lambda\circ \pi_{\Lambda},
\end{align*}
where $\{g_1,g_2\}_\Lambda$ is defined by
\begin{align*}
	\{g_1,g_2\}_{\Lambda}(p,q):=\sum_{l\in\Lambda}\sum_{i=1}^d \bigg(\frac{\partial g_1(p,q)}{\partial q_{l,i}}\frac{\partial g_2(p,q)}{\partial p_{l,i}}-\frac{\partial g_1(p,q)}{\partial p_{l,i}}\frac{\partial g_2(p,q)}{\partial q_{l,i}}\bigg),
\end{align*}
for all $(p,q)\in\Omega_\Lambda$.

One can prove that $\{g_1\circ \pi_{\Lambda},g_2\circ \pi_{\Lambda}\}$ does not depend on $\Lambda$ and lands in $\SR(\Omega)$; the essential part of the proof is that a partial derivative of a levee is again a levee.
Furthermore, the Poisson bracket thus defined coincides with the Poisson bracket introduced in \cite{vanNuland_2019}.


\subsection{Local Hamiltonians}\label{subsection: Hamiltonian}
For each finite $\Lambda\Subset\Gamma$ we consider the local Hamiltonian 
\begin{align}\label{Hamiltoniannew2}
H_\Lambda(p,q):=\sum_{k\in\Lambda}\left(\frac{\|p_k\|^2}{2m_k}+\frac{\nu_k\|q_k\|^2}{2}\right)+\sum_{k,l\in\Lambda}V_{kl}(q_k-q_l),
\end{align}
for $(p,q)\in\Omega$. 
Here, $\|\cdot\|$ is the Euclidean norm on $\R^d$,  and $m_k>0$ and $\nu_k> 0$ denote the mass and force constant of particle $k$, and $V_{kl}$ denotes the interaction potential between particles $k$ and $l$, subject to conditions below. 
We note that $H_\Lambda(p,q)$ depends solely on $(p_\Lambda,q_\Lambda)=\pi_\Lambda(p,q)\in\Omega_\Lambda$, and hence we may view $H_\Lambda$ as a function acting on the finite-dimensional phase space $\Omega_\Lambda$.
Observe furthermore that the model defined by \eqref{Hamiltoniannew2} can be interpreted as a generalization of an oscillating and interacting lattice system. 

For a multi-index $\beta:\{1,\cdots, d\}\to\mathbb{Z}_{\geq 0}$ with $|\beta|=\sum_{i=1}^{d}\beta(i)$,  we write
$$\partial^\beta:=\partial_{1}^{\beta(1)}\cdots\partial_{d}^{\beta(d)},$$
where $\partial_{i}^{\beta(i)}=\partial^{\beta(i)}/\partial x_i^{\beta(i)}$ are the usual partial derivatives of order $\beta(i)$ corresponding to the $i^{th}$ coordinate of $\mathbb{R}^d$.
\begin{assumption}\label{conditions}
The following conditions are assumed:
\begin{itemize}
    \item[(i)] $V_{kl}(x)=V_{lk}(-x)$;
\item[(ii)] $V_{kl}\in C_0^{\infty}(\mathbb{R}^d,\bR)$ for each $k,l\in\Gamma$;
\item[(iii a)] there exist constants $C\geq 0$ and $C_{kl}\geq 0$ for each $k,l\in\Gamma$ such that
$$\|\partial^\beta V_{kl}\|_\infty\leq C_{kl} C^{|\beta|}$$
for all $\beta:\{1,\ldots,d\}\to\Z_{\geq 0}$;
\item[(iii b)] $\sup_{k\in\Gamma}\sum_{l\in\Gamma}C_{kl}<\infty$;
\item[(iv)]  $\sup_{k\in\Gamma}\{m_k\nu_k,1/(m_k\nu_k)\}<\infty$.
\end{itemize}
\end{assumption}
The summability condition (iii b) is an abstraction and generalisation of the stability condition of \cite{Vuillermot_1980}. 
\begin{remark}
Suppose that all interactions are given by the same symmetric potential $V$.
Let $x_{kl}$ be the vector from the pin of $k$ to the pin of $l$. Then automatically $x_{kl}=-x_{lk}$. If we define
$$V_{kl}(x):=\tfrac12 V(x-x_{kl}),$$
it follows that
$$V_{kl}(x)=\tfrac12 V(x-x_{kl})=\tfrac12 V(-x+x_{kl})=\tfrac12 V(-x-x_{lk})=V_{lk}(-x),$$
which is precisely (i). Condition (iii) has to be interpreted in the sense that for each particle $k$, the combined influence of all other particles $l$ sums to something finite. This can be achieved when $\Gamma$ has a group structure by setting $V_{kl}(x)=\tfrac12 j(k-l)V(x-x_{kl})$, and taking $C_{kl}:=\frac{1}{2}j(k-l)$, for a function $j\in \ell^1(\Gamma)$ which, acting as a regulator approximating $j=1_{\Gamma}$, is significantly more general than $j$ being supported on a finite set of neighbours, and like \cite{Vuillermot_1980} signals importance of the $\ell^1$-topology. 
\end{remark}
If one rephrases the prescription  \eqref{Hamiltoniannew2} in terms of the ``interaction potential'' $\Psi$ one finds in the literature \cite{Israel_1979,Landsman}, one obtains up to a constant
 \begin{align}\label{abstract potential}
 	H_\Lambda=\sum_{\Lambda'\subseteq\Lambda}\Psi_{\Lambda'},\qquad\Psi_{\Lambda'}(p,q)=\begin{cases}
 		\frac{||p_k||^2}{2m_k}+\frac{\nu_k||q_k||^2}{2}, \ \ \ \text{if} \ \Lambda'=\{k\};\\
 		2V_{kl}(q_k-q_l), \ \ \ \text{if} \ \Lambda'=\{k,l\};\\
 		0, \ \ \ \text{otherwise}.
 	\end{cases}
 \end{align}
\subsection{Definition of the dynamics on $\CR(\Omega_\Lambda)$}\label{classical dynamical system resolvent}
For finite subsystems $\Omega_\Lambda$ we identify $\Omega_\Lambda\simeq\R^{|\Lambda|d}\times\R^{|\Lambda|d}$. Assumption (iiia) on $V_{kl}$ together with the Fundamental Theorem of Calculus implies Lipschitz continuity of the $\partial_i V_{kl}$: $$\|\partial_i V_{kl}(x)-\partial_i V_{kl}(y)\|\leq\sum_{j=1}^d\int_y^x \|\partial_{j}\partial_i V_{kl}\|_\infty\leq  dC_{kl}C^2\|x-y\|.$$ Therefore, by the Picard--Lindel\"{o}f theorem, the
Hamilton equations have unique solutions on $\Omega_\Lambda$. More precisely,
for every $(p,q)\in\Omega_\Lambda$, there exists a unique function $\R\to\Omega_\Lambda,$ $t\mapsto\Phi^t_{H_\Lambda}(p,q)$ such that
\begin{align}
    \frac{d}{dt}\Phi_{H_\Lambda}^t(p,q)_{k,i}&=\bigg(-\frac{\partial}{\partial q_{k,i}}H_\Lambda(\Phi_{H_\Lambda}^t(p,q)),\frac{\partial}{\partial p_{k,i}} H_\Lambda(\Phi_{H_\Lambda}^t(p,q))\bigg)\label{hamilton equations};\\
    \Phi_{H_\Lambda}^0(p,q)&=(p,q)\label{hamilton equations8},
\end{align}
for all $k\in\Lambda$, $i=1,\cdots,d$; and similarly one defines $\Phi_H^t$ for other Hamiltonians $H$.  The expression on the right-hand side of the first line of equation \eqref{hamilton equations} are simply the components of the Hamiltonian vector field $X_{H_\Lambda}$ of $H_\Lambda$ evaluated at $\Phi^t_{H_\Lambda}(p,q)\in \Omega_\Lambda$. It is a general fact that $\Phi_{H_\Lambda}^t:\Omega_\Lambda\to\Omega_\Lambda$ is a homeomorphism.  

To use a compact notation, we indicate by $\nabla_q$ the gradient 
\begin{align}
    \nabla_q=\sum_{k\in\Lambda}\sum_{i=1}^d(e_k\otimes \delta_i)\frac{\partial}{\partial q_{k,i}},
\end{align}
where $\{e_k\}_{k\in\Lambda}$ a basis for $\mathbb{R}^{|\Lambda|}$ and $\{\delta_i\}_{i=1}^d$ a basis for $\mathbb{R}^d$,
and 
we define $\nabla_p$ in a similar way.  In contrast to the single-site gradient $\nabla$, the gradients $\nabla_p$ and $\nabla_q$ depend on the finite system $\Lambda\Subset\Gamma$.
It follows that equation \eqref{hamilton equations} now reads
\begin{align*}
     \frac{d}{dt}\Phi_{H_\Lambda}^t(p,q)&= \bigg(-\nabla_qH_\Lambda(\Phi_{H_\Lambda}^t(p,q)),\nabla_pH_\Lambda(\Phi_{H_\Lambda}^t(p,q))\bigg).
\end{align*}
The following definition connects the above time evolution to the algebraic setting.

\begin{definition}[Time evolution in algebraic sense]\label{def:time}
Let $\Lambda\Subset\Gamma$ be finite and consider the Hamiltonian $H_\Lambda$ of \eqref{Hamiltoniannew2}. For each fixed $f\in \CR(\Omega_\Lambda)$ the (algebraic) local time evolution is defined to be the pull-back of the flow, that is, the map $\mathbb{R}\ni t \mapsto (\Phi_{H_\Lambda}^t)^*(f)$. 
For each $t\in\mathbb{R}$ we also define
\begin{align}
    &\alpha_\Lambda^t:\pi_{\Lambda}^*\CR(\Omega_\Lambda)\to \pi_{\Lambda}^*\CR(\Omega_\Lambda);\\
    &\alpha_\Lambda^t(f\circ\pi_\Lambda):=f\circ\Phi_{H_\Lambda}^t\circ\pi_\Lambda, \ \  (f\in\CR(\Omega_\Lambda)).
\end{align}
Since $\CR(\Omega_\Lambda)\simeq\pi_\Lambda^*\CR(\Omega_\Lambda)$ the local dynamics can be seen as a map acting on functions defined on all of $\Omega$, which is a convenient property for the forthcoming discussions in Sections \ref{Section: infinite dynamics}--\ref{Sec:strong continuity}.
\end{definition}

\section{Finite dimensional dynamical systems}\label{Sect:finite cases}
The key objective of this paper is to show that the resolvent algebra $\CR(\Omega)$ is stable under the dynamics induced by the local Hamiltonians \eqref{Hamiltoniannew2} and the corresponding thermodynamic limit.
In this section we obtain stability of $\CR(\Omega_\Lambda)$, for a fixed finite subsystem $\Lambda\Subset\Gamma$. The extension of this finite result to all of $\Gamma$ will be the content of Section \ref{Section: infinite dynamics}.

\subsection{Noninteracting time evolution}
To prove time-stability of $\CR(\Omega_\Lambda)$ under $H_\Lambda$ (Theorem \ref{thm: main}) we firstly prove time-stability under the noninteracting Hamiltonian $H_\Lambda^0$, where
\begin{align}
& H_\Lambda=H_{\Lambda}^{0}+V_{\Lambda};\label{ham 3}\\
	&H_{\Lambda}^{0}(p,q):=\sum_{k\in\Lambda}\Big(\frac{||p_k||^2}{2m_k}+\frac{\nu_k||q_k||^2}{2}\Big);\label{free hamiltonian}\\
	&V_\Lambda(p,q):=\sum_{k,l\in\Lambda}V_{kl}(q_k-q_l)\label{potentials}.
\end{align}
\begin{lemma}\label{lemm: simple time evolution}
    The time evolution of the simple harmonic oscillator $H_{\Lambda}^{0}$ of the finite particle system $\Lambda\Subset\Gamma$ preserves the commutative resolvent algebra, i.e.,
    \begin{align}\label{preservation of dynamics: simple case}
        (\Phi_{H_\Lambda^{0}}^t)^*(\CR(\Omega_\Lambda))= \CR(\Omega_\Lambda),
    \end{align}
   for each $t\in\bR$. Moreover, $(\Phi_{H_\Lambda^{0}}^t)^*(\SR(\Omega_\Lambda))=\SR(\Omega_\Lambda).$
\end{lemma}
\begin{proof}
The Hamiltonian $H_{\Lambda}^{0}$ can be written as
$$H_{\Lambda}^{0}(p,q)=\sum_{k\in\Lambda}\sum_{i=1}^d\left(\frac{|p_{k,i}|^2}{2m_k}+\frac{\nu_k|q_{k,i}|^2}{2}\right).$$
The corresponding Hamiltonian system is integrable, indeed, it can be decomposed into a Cartesian product of $|\Lambda|d$ independent subsystems with Hamiltonians
\begin{align}\label{reduction}
    H_{k,i}(p_{k,i},q_{k,i})=\frac{|p_{k,i}|^2}{2m_k}+\frac{\nu_k|q_{k,i}|^2}{2}.
\end{align}
Each such subsystem corresponds to a simple harmonic oscillator of a one-dimensional one-particle system with associated flows $\Phi_{k,i}^t$, for $k\in\Lambda$, $i\in\{1,\ldots,d\}$. The flow of the total system is given by the composition of the individual flows $\Phi_{k,i}^t$, each acting as the identity on all but one factor in the Cartesian product.
The flows of these subsystems are well-known: up to scaling they are given by rotations on two-dimensional phase space with angular frequency $\omega_{k,i}:=\sqrt{\nu_k/m_k}$ for each $i=1,\cdots d$.
Therefore $\Phi_{H_\Lambda^0}^t$ is a composition of linear maps, hence linear itself. Its pull-back $(\Phi_{H_\Lambda^0}^t)^*$ therefore sends any levee to a levee.
Since the levees are dense in $\CR(\Omega_\Lambda)$, the lemma follows.
%
\end{proof}
\noindent
\subsection{Time evolution for compactly supported potentials}
The step extending noninteracting time evolution to time evolution for compactly supported interaction potentials ($V_{kl}\in C^\infty_\c(\R^d)$) is by far the most technical step, and will culminate in the following proposition.
\begin{proposition}\label{prop: compact case}
Let $\Lambda\Subset\Gamma$ be finite, and assume that $V_{kl}\in C_\c^\infty(\R^d)$ for all $k,l\in\Lambda$. Then, for all $t\in[0,1]$,
    \begin{align}\label{eq:to prove t_final}
(\Phi_{H_\Lambda}^{t})^*(\CR(\Omega_\Lambda))\subseteq \CR(\Omega_\Lambda).
\end{align}
It follows that, for all $t\in\R$,
\begin{align}\label{eq:to prove t_final2}
(\Phi_{H_\Lambda}^{t})^*(\CR(\Omega_\Lambda))= \CR(\Omega_\Lambda).
\end{align}
\end{proposition}
We explain how the second assertion follows from the first. If indeed we can prove \eqref{eq:to prove t_final} for all $t\in[0,1]$, then we have 
 $$(\Phi_{H_\Lambda}^{Mt})^*(\CR(\Omega_\Lambda))=((\Phi_{H_\Lambda}^{t})^*)^M(\CR(\Omega_\Lambda))\subseteq \CR(\Omega_\Lambda)$$
 for any $M\in\mathbb{Z}_{\geq 0}$, implying \eqref{eq:to prove t_final} for all $t\in\R_{\geq0}$, and by substituting $(-p,q)$ for $(p,q)$ and using invariance of $\CR(\Omega_\Lambda)$ under pull-backs of linear maps we obtain \eqref{eq:to prove t_final} for all $t\in\mathbb R$. Noticing that $(\Phi_{H_\Lambda}^{-t})^*$ is the inverse of $(\Phi_{H_\Lambda}^{t})^*$, we obtain \eqref{eq:to prove t_final2} for all $t\in\R$ as well.


The proof of the first assertion of Proposition \ref{prop: compact case} consists of several steps for which the following notation is introduced.
We write
\begin{align}
\pi_{kl}:\Omega_\Lambda\to\R^d;\\
\pi_{kl}(p,q)=q_k-q_l,
\end{align}
and introduce the potentials
\begin{align}\label{general potential resolvent algebra}
    V_\mathcal{N}(p,q):=\sum_{\{k,l\}\in\mathcal{N}}2V_{kl}(\pi_{kl}(p,q)),
\end{align}
for all (finite) sets
\begin{align}\label{eq:mN}
\mathcal{N}\subseteq\{\{k,l\}:~k,l\in\Lambda,k\neq l\}.
\end{align}
We now fix such a set $\mN$. For the proofs in this section we shall consider the general Hamiltonians defined according to $V_\mN$:
\begin{align*}
H_{\mN}:=&H_\Lambda^{0}+V_\mN;\\
 H_{\mathcal{N}-kl}:=&H_{\mN\setminus\{\{k,l\}\}}\\
 =&H_\Lambda^{0}+V_\mN-2V_{kl}\circ \pi_{kl},
\end{align*}
for all $\{k,l\}\in\mathcal{N}$.


\begin{definition}\label{def:integral curves}
Let $\{k,l\}\in\mathcal{N}$, and consider a given $(p_0,q_0)\in\Omega$. We denote 
\begin{align*}
    \omega(t)&\equiv (p(t),q(t)):=\Phi^t_{H_\mathcal{N}}(p_0,q_0);\\
    \tilde{\omega}(t)&\equiv (\tilde p(t),\tilde q(t)):=\Phi^t_{H_{\mathcal{N}-kl}}(p_0,q_0);\\
    \omega^0(t)&\equiv(p^0(t),q^0(t)):=\Phi^t_{H_\Lambda^{0}}(p_0,q_0).
\end{align*}
These are the integral curves through $(p_0,q_0)$ corresponding to three different flows. 
\end{definition}

In the proof of Proposition \ref{prop: compact case}, we quantitatively compare the flow of $H_\mN$ with the flow of $H_{\mN-kl}$, and thus build up the full $H_\Lambda$ inductively from $H^0_\Lambda$, an idea which is based on \cite[Prop.11]{vanNuland_Stienstra_2019}. However, both because our configuration space is noncompact, and because our noninteracting Hamiltonian is more complicated, we will need ideas beyond \cite{vanNuland_Stienstra_2019}. The first is captured by the following lemma, which states that two noninteracting particles oscillating with possibly different frequencies do not spend too much time in each other's presence. 
For this purpose we introduce the following notation. We write
$$\pi_{kl}^{\mom,\pos}=(\pi_{kl}^\mom,\pi_{kl}^\pos):\Omega_\Lambda^\mom\oplus\Omega_\Lambda^\pos\to\R^{d}\oplus\R^d,$$
where 
$$\pi_{kl}^\mom(p)=\frac{1}{m_k}p_k-\frac{1}{m_l}p_l,\qquad\pi_{kl}^\pos(q)=q_k-q_l.$$ 

\begin{lemma}\label{Lemm: measure zero}
Let $\lambda$ denote the Lebesgue measure on $[0,1]$, let $B_R(0)$ be the ball around zero for some radius $R>0$, and fix $k,l\in\Lambda\Subset\Gamma$. 
\begin{enumerate}
    \item Suppose that $\frac{\nu_k}{m_k}= \frac{\nu_l}{m_l}$. For every $\epsilon>0$ there exists $D_{k,l,\epsilon}>0$ such that for all $(p_0,q_0)\in\Omega_\Lambda$ satisfying $\|\pi^{\mom,\pos}_{kl}(\omega)\|>D_{k,l,\epsilon}$ we have (in the notation of Definition \ref{def:integral curves})
\begin{align}\label{eq:measure epsilon}
\lambda\left(\left\{t\in[0,1]\mid q^0_k(t)-q^0_l(t)\in B_R(0)\right\}\right)<\epsilon.
\end{align}
    \item Suppose that $\frac{\nu_k}{m_k}\neq \frac{\nu_l}{m_l}$. For every $\epsilon>0$ there exists $D_{k,l,\epsilon}>0$ such that for all $(p_0,q_0)\in\Omega_\Lambda$ satisfying $\|(\omega_k,\omega_l)\|>D_{k,l,\epsilon}$ we have (in the notation of Definition \ref{def:integral curves})
$$\lambda\left(\left\{t\in[0,1]\mid q^0_k(t)-q^0_l(t)\in B_R(0)\right\}\right)<\epsilon.$$
\end{enumerate}
\end{lemma}
\begin{proof}
We may throughout the proof assume that $\Lambda=\{k,l\}$. Moreover, we leave out the superscript $0$ referring to the noninteracting dynamics. The whole proof is in terms of the flow of $H^0_{\{k,l\}}$.

 \paragraph{Proof of 1.}
Suppose the statement does not hold, so there exists an $\epsilon>0$ such that for all $D$ there exist $(p_0,q_0)\in\Omega_\Lambda$ satisfying 
\begin{align}\label{eq:ongerijmde}
\sqrt{\Big\|\frac{(p_0)_k}{m_k}-\frac{(p_0)_l}{m_l}\Big\|^2+\|(q_0)_k-(q_0)_l\|^2}>D
\end{align}
and such that \eqref{eq:measure epsilon} does not hold.
 The function $\R\to\R^d$, $t\mapsto q_k(t)-q_l(t)$ is continuous and periodic with some period independent of the initial conditions, so the set in the argument of $\lambda$ in \eqref{eq:measure epsilon} contains a nontrivial interval $[a,b]$ of length $|b-a|\geq N\epsilon$, for a number $N$ independent of $(p_0,q_0)$. For all $t\in [a,b]$ we have
 $\|q_k(t)-q_l(t)\|\leq R$,
 so that on account of \eqref{hamilton equations}
\begin{align*}
\|(\frac{p_k(t)}{m_k}-\frac{p_l(t)}{m_l})-(\frac{p_k(a)}{m_k}-\frac{p_l(a)}{m_l})\|&=\|\int_a^t(\frac{\nu_kq_k(s)}{m_k}-\frac{\nu_lq_l(s)}{m_l})ds\|\\
&\leq \frac{\nu_k}{m_k} R|b-a|,
\end{align*}
 for all $t\in[a,b]$. Hence, by twice applying the reverse triangle inequality,
 \begin{align*}
&\|q_k(b)-q_l(b)\|\\
&\geq\|\int_a^b\big(\frac{p_k(t)}{m_k}-\frac{p_l(t)}{m_l}\big)dt\|-\|q_k(a)-q_l(a)\|\\
&\geq\|\int_a^b\big(\frac{p_k(a)}{m_k}-\frac{p_l(a)}{m_l}\big)dt\|-\|\int_a^b\Big(\big(\frac{p_k(t)}{m_k}-\frac{p_l(t)}{m_l}\big)-\big(\frac{p_k(a)}{m_k}-\frac{p_l(a)}{m_l}\big)\Big)dt\|\\
&\quad-R\\
&\geq|b-a|\|\frac{p_k(a)}{m_k}-\frac{p_l(a)}{m_l}\|-|b-a|^2\frac{\nu_k}{m_k}R-R.
 \end{align*}
We obtain the bound
$$\Big\|\frac{p_k(a)}{m_k}-\frac{p_l(a)}{m_l}\Big\|\leq \frac{2R+|b-a|^2\frac{\nu_k}{m_k}R}{|b-a|}\leq \frac{2R}{\epsilon N}+|b-a|\frac{\nu_k}{m_k}R.$$
Hence,
\begin{align}\label{eq:difference energy} \Big\|\frac{p_k(a)}{m_k}-\frac{p_l(a)}{m_l}\Big\|^2+\|q_k(a)-q_l(a)\|^2\leq \big(\frac{2R}{\epsilon N}+\frac{\nu_k}{m_k}R\big)^2+R^2,
 \end{align}
since $|b-a|\leq 1$. 
 Because $(\frac{p_k(t)}{m_k},q_k(t))$ and $(\frac{p_l(t)}{m_l},q_l(t))$ are both solutions of the same simple harmonic oscilator, $(\frac{p_k(a)}{m_k}-\frac{p_l(a)}{m_l},q_k(a)-q_l(a))$ is, up to a constant scaling, a rotation of $(\frac{(p_0)_k}{m_k}-\frac{(p_0)_l}{m_l},(q_0)_k-(q_0)_l)$, so $\|\frac{(p_0)_k}{m_k}-\frac{(p_0)_l}{m_l}\|^2+\|(q_0)_k-(q_0)_l\|^2$ is bounded by the same constant as \eqref{eq:difference energy}, multiplied by another constant independent of $(p_0,q_0)$. But this contradicts our assumption that \eqref{eq:ongerijmde} has to hold for all $D$. We therefore obtain the result.
%

\paragraph{Proof of 2.}
We suppose the statement does not hold. This implies the existence of an $\epsilon>0$ and a sequence $(\omega^j_0)_{j=1}^\infty=((p^j_0,q^j_0))_{j=1}^\infty\subseteq\Omega_\Lambda=\Omega_{\{k,l\}}$, such that
\begin{align}\label{eq:omega difference diverges}
    \lim_{j\to\infty}\sqrt{\|(\omega_0^j)_k\|^2+\|(\omega_0^j)_l\|^2}=\infty
\end{align}
and
$$\lambda\left(\left\{t\in[0,1]\mid (q^j)_k(t)-(q^j)_l(t)\in B_R(0)\right\}\right)\geq\epsilon.$$
As $\R\to\R^d$, $t\mapsto (q^j)_k(t)-(q^j)_l(t)$ is a sum of two harmonic functions with two positive periods that are independent of $j$, the set above (the argument of $\lambda$) is a union of $N$ open intervals $(a^j_1,b^j_1),\ldots,(a^j_N,b^j_N)$, where $N$ does not depend on $j$. We denote $I^j=(a^j_1,b^j_1,\ldots,a^j_N,b^j_N)\in[0,1]^N$. We set
$$c_j:=\|\omega^j_0\|=\sqrt{\|(\omega^j_0)_k\|^2+\|(\omega^j_0)_l\|^2},$$
which we may assume to be nonzero because $c_j\to\infty$.
We also set
$$\bar\omega^j_0:=\frac{1}{c_j}\omega^j_0.$$
According to the Hamiltonian equations for $H^0_{\{k,l\}}$ (two harmonic oscillators) we obtain
$$\bar\omega^j(t)=\frac{1}{c_j}\omega^j(t).$$
It also follows that $\|\bar\omega^j_0\|= 1$.

Because of the above construction, the sequence $((\bar\omega_0^j,I^j))_{j=1}^\infty$ takes values in a compact set, and so admits a convergent subsequence $((\bar\omega_0^{j_i},I^{j_i}))_{i=1}^\infty$ whose limit we denote $(\bar\omega_0,I)$, and we write $\bar\omega_0=(\bar p_0,\bar q_0)$ and $I=(a_1,b_1,\ldots,a_N,b_N)$. We obtain
\begin{align*}
    \lambda(I)\geq\epsilon,
\end{align*}
which implies that there is a nontrivial open interval $(a_m,b_m)$ for some $m\in\{1,\ldots,N\}$. For all $t\in (a_m,b_m)$, we obtain, for $i\in\mathbb{Z}_{\geq 0}$ large enough,
$$\bar q^{j_i}_k(t)-\bar q^{j_i}_l(t)\in \frac{1}{c_{j_i}} B_R(0).$$
As $c_{j_i}\to\infty$, we obtain, for all $t\in(a_m,b_m)$,
$$\bar q_k(t)=\bar q_l(t).$$
As $\bar q_k$ and $\bar q_l$ are both (scaled and shifted) sine functions, the fact that they coincide on a nontrivial interval implies that they are equal everywhere. This implies that their frequencies are equal, which contradicts our assumption.
\end{proof}

We shall make repeated use of the following formulation of Gronwall's inequality, which is well-known (see e.g. \cite{Ralph_Howard}).

 \begin{lemma}[Gronwall's inequality]\label{lem:Gronwall's inequality}
 Let $t_0<t_1\in\R$, let $m\in \mathbb{Z}_{\geq 0}$, and let $U\subset\R^m$ be open.  let $F,G: [t_0,t_1]\times U \to\R^{m}$  be  continuous functions, and let $y:[t_0,t_1]\to \R^m$,  $z:[t_0,t_1]\to \R^m$, satisfy the initial value problems:
 \begin{align*}
     &\frac{d}{dt}y(t)=F(t,y(t)); \ \ y(t_0)=y_0\\
     &\frac{d}{dt}z(t)=G(t,z(t)); \ \ z(t_0)=z_0.
 \end{align*}
 Assume there is a constant  $C\geq 0$ such that
 $$||G(t,x)-G(t,x')||\leq C ||x-x'||,$$
 and a continuous function $\varphi:[t_0,t_1]\to [0,\infty)$ such that
 $$||F(t,y(t))-G(t,y(t))||\leq \varphi(t).$$
 Then, for all $t\in [t_0,t_1]$,
 \begin{align*}
     ||y(t)-z(t)||\leq e^{C|t-t_0|}||y_0-z_0||+e^{C|t-t_0|}\int_{t_0}^te^{-C|s-t_0|}\varphi(s)ds.
 \end{align*}
 \end{lemma}
Lemma \ref{Lemm: measure zero} and Gronwall's inequality enable us to prove the following result. We recall that $\mN\subseteq\Lambda\times\Lambda$ is a fixed set satisfying \eqref{eq:mN}.
\begin{proposition}\label{Prop: fundamental}
For all $\{k,l\}\in\mN$ and all $\epsilon>0$ there exists $D_{k,l,\epsilon}>0$, such that for all $\omega\in\Omega_\Lambda$ satisfying either
\begin{enumerate}[(i)]
    \item\label{item:1 prop Dklepsilon} $\|\pi^{\mom,\pos}_{kl}(\omega)\|>D_{k,l,\epsilon}$ if $\frac{\nu_k}{m_k}=\frac{\nu_l}{m_l}$;
    \item\label{item:2 prop Dklepsilon} $\|(\omega_k,\omega_l)\|>D_{k,l,\epsilon}$ if $\frac{\nu_k}{m_k}\neq\frac{\nu_l}{m_l}$,
\end{enumerate}
we have, for all $t\in[0,1]$,
\begin{align}\label{statement}
    \|\Phi^{t}_{H_\mathcal{N}}(\omega)-\Phi^{t}_{H_{\mathcal{N}-kl}}(\omega)\|<\epsilon.
\end{align}
\end{proposition}

\begin{proof}
We are going to apply Gronwall's inequality to the following data.
We hereto introduce three time independent vector fields, whose components are defined by
\begin{align}
&F:=X_{H_\Lambda^{0}}:(p,q)\mapsto \Big(-\nabla_qH_\Lambda^0(p,q),\nabla_pH_\Lambda^0(p,q)\Big);\label{IVP f}\\
&G:=X_{H_\mN}:(p,q)\mapsto\Big(-\nabla_q H_\mathcal{N}(p,q),\nabla_p H_\mathcal{N}(p,q)\Big);\label{IVP g}\\
&\widetilde{G:}=X_{H_{\mN-kl}}:(p,q)\mapsto \Big(-\nabla_qH_{\mathcal{N}-kl}(p,q),\nabla_pH_{\mathcal{N}-kl}(p,q)\Big)\label{IVP gb},
\end{align}
with $p,q\in \R^{|\Lambda|d}$. As before, we denote the (necessarily unique) flows corresponding to the Hamilton  vector fields above, respectively,  by $\omega^0$, $\omega$, and $\widetilde{\omega}$, and remind that the initial points coincide at time $t=0$, i.e. $\omega(0)=\omega^0(0)=\tilde{\omega}(0)=(p_0,q_0)$. On account of Lemma \ref{lem:Gronwall's inequality} and our conditions on the potential $V_\mathcal{N}$ (cf. \NAK\ref{subsection: Hamiltonian}), we may estimate
\begin{align*}
    &||G(\omega)-G(\omega')||\leq |\Lambda| \max_{k}\Big|\frac{p_k}{m_k}-\frac{{{p_k'}}}{m_k}\Big|+ |\Lambda|\max_{k}|\nu_kq_k-\nu_k{q_k'}|\\ &+\sum_{\{k,l\}\in\mathcal{N}}C_{kl}(||q_k-q_k'||+|q_l-q_l'||)\\
    &\leq C_1||p-p'||_\infty+C_2||q-q'||_\infty + C_3||q-q'||_\infty\\&\leq 
    C_G||\omega-\omega'||,
\end{align*}
for suitable positive constants $C_1,C_2,C_3,C_{kl}$ and $C_G$. 
Similar estimates yield
\begin{align*}
||\widetilde{G}(\omega)-\widetilde{G}(\omega')||\leq C_{\widetilde G}||\omega-\omega'||.
\end{align*}
Notice that 
$$||F(\omega^0(t))-G(\omega^0(t))||_\infty\leq \|\nabla_qV_{\mathcal{N}}\|_\infty<\infty,$$ as $\mathcal{N}$ is finite. The same estimate holds for  $||F(\omega^0(t))-\widetilde G(\omega^0(t))||$. If we then define
$d_1:=||\nabla_q V_\mathcal{N}||_\infty$, we may apply Gronwall's inequality (cf. Lemma \ref{lem:Gronwall's inequality}) and obtain
\begin{align}\label{def d}
    &||\pi^\pos_{kl}(q^0(t))-\pi^\pos_{kl}(q(t))||\leq 2||\omega^0(t)-\omega(t)||\leq 2e^Cd_1=:\Delta;
\end{align}
and notice again that the same estimate holds for $||\pi^\pos_{kl}(q^0(t))-\pi^\pos_{kl}(\tilde q(t))||$.
These estimates allow us to apply Lemma \ref{Lemm: measure zero} and estimate the norm difference $||\omega(t)-\tilde\omega(t)||$, by comparing these flows with $\omega^0(t)$. To this end, we fix an $r>0$ such that  $\supp V_{kl}\subseteq B_r(0)$, the ball around zero with radius $r$. We set $R:=\Delta+r$, where $d$ is defined through \eqref{def d}. It follows that 
\begin{align}
    &t\notin E_{kl}:=\left\{t\in[0,1]\mid \pi^\pos_{kl}(q^0(t))\in B_R(0)\right\}\Rightarrow\nonumber \\ &V_{kl}(\pi^\pos_{kl}(q(t)))=V_{kl}(\pi^\pos_{kl}(\tilde{q}(t)))=0.\label{implication}
\end{align}

In order to estimate both flows, we employ Lemma \ref{Lemm: measure zero}. For the current $k,l\in\Gamma$, the above defined $R$, and a given $\epsilon>0$, Lemma \ref{Lemm: measure zero} provides a $D_{k,l,\epsilon}>0$ such that, for $(p_0,q_0)$ under either assumption \eqref{item:1 prop Dklepsilon} or \eqref{item:2 prop Dklepsilon}, we have 
$$\lambda(E_{kl})<\epsilon.$$
In order to apply Gronwall's inequality, we consider the characteristic function $\chi_{E_{kl}}$ of $E_{kl}$. By standard results, there exists a continuous function $\tilde\chi_{kl}\geq\chi_{kl}$ such that $\int_0^1\tilde\chi_{kl}<2\epsilon$. For $t\in E_{kl}$, $\tilde\chi_{kl}(t)\geq 1$, and if $t\notin E_{kl}$, it follows on account of $\eqref{implication}$ that $||\nabla_q V_{kl}(q_k(t)-q_l(t))||=0$. Hence, for any $t\in[0,1]$ we have
\begin{align*}
||G(\omega(t))-\widetilde G(\omega(t))||&\leq 2||\nabla_q V_{kl}(q_k(t)-q_l(t))|| \\
&\leq 4||\nabla V_{kl}||_\infty \tilde\chi_{kl}(t).
\end{align*}
We now let 
$$\varphi(t):=4||\nabla V_{kl}||_\infty \tilde\chi_{kl}(t),$$ we can apply Lemma \ref{lem:Gronwall's inequality}, using $\omega(0)=\omega^0(0)=\tilde{\omega}(0)$, and find
\begin{align*}
    ||\omega(t)-\widetilde{\omega}(t)||\leq& e^{C_Gt}\int_0^te^{-C_Gs}\varphi(s)ds\\
    \leq& 4\|\nabla V_{kl}\|_\infty e^{C_G}\int_0^t\tilde\chi_{kl}(s)ds\\
    \leq&8\|\nabla V_{kl}\|_\infty e^{C_G}\epsilon.
\end{align*}
By rescaling $\epsilon$ we find $||\omega(t)-\widetilde{\omega}(t)||<\epsilon,$ as desired.
\end{proof}

The final step in the proof is to cover phase space with open regions on which $\Phi_{H_\mN}$ has an understood behavior. The open regions need to be chosen in a ``linear'' fashion, in order to obtain a subordinate partition of unity consisting of elements of $\CR(\Omega_\Lambda)$. Thus we will patch together functions in $\CR(\Omega_\Lambda)$ making something that approximates $g\circ \Phi_{H_\mN}^{t}$ for an arbitrary $g\in \CR(\Omega_\Lambda)$.

In the following, we fix $g\in\CR(\Omega_\Lambda)$ and $\epsilon>0$. We also fix $\delta>0$
such that $$||\omega-\omega'||<\delta\Rightarrow |g(\omega)-g(\omega')|< \epsilon\qquad (\omega,\omega'\in\Omega_\Lambda),$$ which can be done as $\CR(\Omega_\Lambda)$ is inside the uniformly continuous functions. 

Proposition \ref{Prop: fundamental} is translated into algebraic language as follows. For every $\{k,l\}\in\mathcal{N}$, we may fix a $D_{k,l,\delta}\equiv D_{kl}>0$ (whose dependence on $\delta$ is omitted) with the property that
\begin{align}\label{uniform continuity}
    \sup_{\omega\in U_{kl}}|(\Phi_{H_\mN}^{t})^*g(\omega)-(\Phi_{H_{\mN-{kl}}}^{t})^*g(\omega)|< \epsilon,
\end{align}
where $U_{kl}$ is defined by 

\begin{align}\label{eq:Ukl}
U_{kl}&:=\begin{cases}
\{\omega\in\Omega_\Lambda:~\|\pi_{kl}^{\mom,\pos}(\omega)\|>D_{kl}\},\qquad&\text{if $\frac{\nu_k}{m_k}=\frac{m_l}{\nu_l}$}\,;\\
\{\omega\in\Omega_\Lambda:~\|(\omega_k,\omega_l)\|>D_{kl}\},\qquad&\text{if $\frac{\nu_k}{m_k}\neq\frac{m_l}{\nu_l}$}\,.
\end{cases}
\end{align}
We moreover define
\begin{align}\label{eq:Wkl}
W_{kl}&:=\begin{cases}
\{\omega\in\Omega_\Lambda:~\|\pi_{kl}^{\mom,\pos}(\omega)\|>2D_{kl}\},\qquad&\text{if $\frac{\nu_k}{m_k}=\frac{m_l}{\nu_l}$}\,;\\
\{\omega\in\Omega_\Lambda:~\|(\omega_k,\omega_l)\|>2D_{kl}\},\qquad&\text{if $\frac{\nu_k}{m_k}\neq\frac{m_l}{\nu_l}$}\,.
\end{cases}
\end{align}
as well as
\begin{align*}
&U_{\infty}:=\left\{\omega\in\Omega_\Lambda:\hspace*{-0.1cm}\begin{array}{l}
\|\pi_{kl}^{\mom,\pos}(\omega)\|<4D_{kl}\textnormal{ for all }\{k,l\}\in\mathcal N\text{ with $\frac{\nu_k}{m_k}=\frac{m_l}{\nu_l}$}\\
\|(\omega_k,\omega_l)\|<4D_{kl}\textnormal{ for all }\{k,l\}\in\mathcal N\text{ with $\frac{\nu_k}{m_k}\neq\frac{m_l}{\nu_l}$}
\end{array}\hspace*{-0.1cm}\right\};\\\\
&W_{\infty}:=\left\{\omega\in\Omega_\Lambda:\hspace*{-0.1cm}\begin{array}{l}
\|\pi_{kl}^{\mom,\pos}(\omega)\|<3D_{kl}\textnormal{ for all }\{k,l\}\in\mathcal N\text{ with $\frac{\nu_k}{m_k}=\frac{m_l}{\nu_l}$}\\ 
\|(\omega_k,\omega_l)\|<3D_{kl}\textnormal{ for all }\{k,l\}\in\mathcal N\text{ with $\frac{\nu_k}{m_k}\neq\frac{m_l}{\nu_l}$}
\end{array}\hspace*{-0.1cm}\right\}.
\end{align*}

Note that $\{U_{Y}\}_{Y\in\mathcal{N}\cup\{\infty\}}$ and $\{W_{Y}\}_{Y\in\mathcal{N}\cup\{\infty\}}$ are open covers satisfying $W_{Y}\subset U_{Y}$ for all $Y\in\mathcal{N}\cup\{\infty\}$.
The approximate behavior of $g\circ\Phi^{t}_{H_\mN}$ (recall that $g\in \CR(\Omega_\Lambda)$ is fixed) on $\cup_{Y\in\mathcal{N}}U_Y$ is described by \eqref{uniform continuity}, and the behavior of  $g\circ\Phi^{t}_{H_\mN}$ on $U_\infty$ is described by the following lemma.

\begin{lemma}\label{Lemm: existence function}
    There exists an $f_\infty\in \CR(\Omega_\Lambda)$ that equals $g\circ\Phi^{t}_{H_\mN}$ on $U_\infty$.
\end{lemma}
\begin{proof}
For all $k,l\in\Lambda$, let us write $k\sim l$ if there exists a `path' of $N\in\mathbb{Z}_{\geq 0}$ connections $\{k_1,k_2\},\{k_2,k_3\},\ldots,\{k_{N-1},k_N\}\in\mN$ such that $k_1=k$ and $k_N=l$. We can then write $\Lambda=\Lambda_0\cup\Lambda_1\cup\cdots\cup\Lambda_M$ for disjoint subsets $\Lambda_0,\ldots,\Lambda_M$ that satisfy:
\begin{itemize}
    \item $k\nsim l$ for all $k\in\Lambda_i,l\in\Lambda_j$, $i\neq j$;
    \item for all $k,l\in\Lambda_j$, $j\geq1$, we have $k\sim l$ and $\frac{\nu_k}{m_k}=\frac{\nu_l}{m_l}$;
    \item for all $k\in\Lambda_0$ there exists an $l\in\Lambda$ such that $k\sim l$ and $\frac{\nu_k}{m_k}\neq\frac{\nu_l}{m_l}$.
\end{itemize}
By definition, the Hamiltonian $H_\mN$ contains no interactions between $k\in\Lambda_j$, $l\in\Lambda_{j'}$ for distinct $j,j'\in\{0,\ldots,M\}$. The time flow therefore splits as 
\begin{align}\label{eq:splitting of flow in Lambda_j's}
\Phi_{H_\mN}^t(\omega)=(\Phi_{H_{\mN,0}}^t(\pi_{\Lambda_0}(\omega)),\ldots,\Phi_{H_{\mN,M}}^t(\pi_{\Lambda_M}(\omega)))
\end{align} 
on $\Omega_\Lambda=\Omega_{\Lambda_0}\times\cdots\times\Omega_{\Lambda_M}$, where $H_{\mN,j}$ is the restriction of $H_{\mN}$ to $\Omega_{\Lambda_j}$, and consequently $\Phi_{H_{\mN,j}}^t$ is the restriction of $\Phi_{H_\mN}^t$ to $\Omega_{\Lambda_j}$.
Moreover, for a suitably large constant $C>0$ (depending on the numbers $D_{kl}$) we can define 
$$\tilde U_\infty:=\left\{\omega\in\Omega_\Lambda:~
\begin{array}{ l }
\|\pi_{kl}^{\mom,\pos}(\omega)\|<C\text{ for all $k,l\in\Lambda_j$, $j\geq1$},\\
\|\omega_k\|<C\text{ for all $k\in\Lambda_0$}
\end{array}
\right\},$$
so that $U_\infty\subseteq\tilde U_\infty$, as we can apply the triangle inequality $\|\pi_{k_1k_3}^{\mom,\pos}(\omega)\|\leq\|\pi_{k_1,k_2}^{\mom,\pos}(\omega)\|+\|\pi_{k_2,k_3}^{\mom,\pos}(\omega)\|$ at each step on the path from $k\in\Lambda_j$ to $l\in\Lambda_j$, $j\geq1$, and the path lengths are uniformly bounded by $|\Lambda|$. It therefore suffices to prove the lemma for $\tilde U_\infty$ instead of $U_\infty$.
The set $\tilde U_\infty$ splits as
\begin{align*}
\tilde U_\infty=U_{\infty,0}\times U_{\infty,1}\times\cdots\times U_{\infty,M}
\end{align*}
on $\Omega_\Lambda=\Omega_{\Lambda_0}\times\cdots\times\Omega_{\Lambda_M}$, where $\overline{U_{\infty,0}}$ is a compact subset of $\Omega_{\Lambda_0}$ 
and, for $j\geq1$,
\begin{align*}
U_{\infty,j}:=&\{\omega\in\Omega_{\Lambda_j}:~\|\pi_{kl}^{\mom,\pos}(\omega)\|<C\text{ for all }k,l\in\Lambda_j\}.
\end{align*}

For all $j\geq1$ we introduce the linear subspace (of which $U_{\infty,j}$ is a neighborhood)
\begin{align*}
 T_j:=&\{\omega\in \Omega_{\Lambda_j}~:~ \pi_{kl}^{\mom,\pos}(\omega)=0\text{ for all }k,l\in\Lambda_j\},\\
 =&\{(p,q)\in\Omega_{\Lambda_j}~:~\frac{p_k}{m_k}=\frac{p_l}{m_l},~ q_k=q_l\forall k,l\in\Lambda_j\}\\
 =&\{p\in\Omega_{\Lambda_j}^\mom~:~\frac{p_k}{m_k}=\frac{p_l}{m_l}~\forall k,l\in\Lambda_j\}\oplus\{q\in\Omega_{\Lambda_j}^\pos~:~ q_k=q_l~\forall k,l\in\Lambda_j\}\\= &T_{j}^\mom\oplus T_{j}^\pos,
\end{align*}
where we have defined $T_{j}^\mom:=\{p\in\Omega_{\Lambda_j}^\mom~:~\frac{p_k}{m_k}=\frac{p_l}{m_l}~\forall k,l\in\Lambda_j\}$ and $T_{j}^\pos=\{q\in\Omega_{\Lambda_j}^\pos~:~ q_k=q_l~\forall k,l\in\Lambda_j\}$.
Let $|m|:=\sum_{k\in\Lambda_j}m_k$. We define an idempotent $\pi_{T_j}:\Omega_{\Lambda_j}\to T_j$ by $\pi_{T_j}:=\pi_{T_j}^\mom\oplus\pi_{T_j}^\pos$ and
\begin{align}
\pi_{T_j}^\mom(p)&:=\frac{1}{|m|}\sum_{k,l\in\Lambda_j}m_ke_k\otimes p_l\\
\pi_{T_j}^\pos(q)&:=\frac{1}{|m|}\sum_{k,l\in\Lambda_j}e_k\otimes m_lq_l.
\end{align}
It is straightforward to check that $\ran\pi_{T_j}\subseteq T_j$ and that $\pi_{T_j}(p,q)=(p,q)$ for $(p,q)\in T_j$, therefore showing that $\pi_{T_j}$ is a linear idempotent. We also introduce
\begin{align*}
S_j:=\left\{(p,q)\in\Omega_{\Lambda_j}~:~\sum_{k\in\Lambda_j}p_k=0,~\sum_{k\in\Lambda_j}q_k=0\right\}.
\end{align*}
We define an idempotent $\pi_{S_j}:\Omega_{\Lambda_j}\to S_j$ by $\pi_{S_j}:=\pi_{S_j}^\mom\oplus\pi_{S_j}^\pos$ and
\begin{align}
\pi_{S_j}^\mom(p)&:=p-\frac{1}{|m|}\sum_{k,l\in\Lambda_j}m_ke_k\otimes p_l\\
\pi_{S_j}^\pos(q)&:=q-\frac{1}{|m|}\sum_{k,l\in\Lambda_j}e_k\otimes m_lq_l.
\end{align}
As before, one can show that indeed $\ran \pi_{S_j}\subseteq S_j$ and $\pi_{S_j}(p,q)=(p,q)$ for $(p,q)\in S_j$, showing that $\pi_{S_j}$ is also a linear idempotent. It is now an easy computation to check that
\begin{align}\label{eq:S_j+T_j}
S_j\oplus T_j=\Omega_{\Lambda_j}.
\end{align}
Moreover, by a direct computation one observes that
$$\pi_{S_j}X_{H_{\mN,j}}=X_{H_{\mN,j}}\pi_{S_j},$$
where $X_{H_{\mN,j}}:\Omega_{\Lambda_j}\to\Omega_{\Lambda_j}$ denotes the vector field associated to  $H_{\mN,j}$ (cf. \textsection\ref{classical dynamical system resolvent}).
From this it follows that
\begin{align}\label{eq:flow and S_j}
\pi_{S_j}\Phi^t_{H_{\mN,j}}=\Phi^t_{H_{\mN,j}}\pi_{S_j}.
\end{align}
Using the fact that all particles in $T_j$  have the same frequency, a similar computation yields the relation
$$\pi_{T_j}X_{H_{\mN,j}}=X_{H^0_{\mN,j}}\pi_{T_j},$$
and from this it follows that
\begin{align}\label{eq:flow and T_j}
\pi_{T_j}\Phi^t_{H_{\mN,j}}=\Phi^t_{H^0_{\mN,j}}\pi_{T_j}.
\end{align}
We set $S_0:=\Omega_{\Lambda_0}$ and $T_0:=\{0\}\subseteq \Omega_{\Lambda_0}$ and define
$$S:=\bigoplus_{j=0}^M S_j,\qquad T:=\bigoplus_{j=0}^M T_j.$$
It follows from \eqref{eq:S_j+T_j} that $S\oplus T=\Omega_\Lambda$. If we set $\pi_S:={\rm id}_{\Lambda_0}\oplus\pi_{S_1}\oplus\cdots\oplus\pi_{S_M}$ and $\pi_T:=0\oplus\pi_{T_1}\oplus\cdots\oplus\pi_{T_M}$, then $\pi_S$ and $\pi_T$ are linear idempotents with images $S$ and $T$, $\pi_S+\pi_T={\rm id}_{\Omega_\Lambda}$, and from \eqref{eq:flow and S_j}, \eqref{eq:flow and T_j}, and \eqref{eq:splitting of flow in Lambda_j's} it follows that
\begin{align}\label{eq:flow and S and T}
\pi_{S}\Phi^t_{H_\mN}=\Phi^t_{H_\mN}\pi_{S},\qquad \pi_{T}\Phi^t_{H_\mN}=\Phi^t_{H^0_\mN}\pi_{T}.
\end{align}
We note that
$$C_0(S)\hat{\otimes} \CR(T):=\overline\spn\{(g_1\circ\pi_{S})(g_2\circ\pi_{T})~:~g_1\in C_0(S),g_2\in\CR(T)\}$$
is an ideal in $\CR(\Omega_{\Lambda})$.
The next step is to show that
\begin{align}\label{preserved}
    (\Phi_{H_\mN}^{t})^*(C_0(S)\hat{\otimes} \CR(T))\subseteq C_0(S)\hat{\otimes} \CR(T).
\end{align}
To see this, we take $g_1\in C_0(S)$ and $g_2\in\CR(T)$. By using \eqref{eq:flow and S and T} we obtain
\begin{align*}
&(\Phi_{H_{\mN}}^t)^*((g_1\circ\pi_{S})(g_2\circ\pi_{T}))(\omega)\\
&=g_1(\pi_{S}(\Phi_{H_{\mN}}^t(\omega)))g_2(\pi_{T}(\Phi_{H_{\mN}}^t(\omega)))\\
&=g_1(\Phi_{H_{\mN}}^t(\pi_{S}(\omega)))g_2(\Phi_{H^0_{\mN}}^t(\pi_{T}(\omega)))\\
&=((\Phi_{H_{\mN}}^t)^*g_1)(\pi_{S}(\omega))((\Phi_{H^0_{\mN}}^t)^*g_2)(\pi_{T}(\omega))\\
&=(((\Phi_{H_{\mN}}^t)^*g_1)\circ\pi_{S})(((\Phi_{H^0_{\mN}}^t)^*g_2)\circ\pi_{T})(\omega).
\end{align*}
As $(\Phi_{H_{\mN}}^t)^*g_1\in C_0(S)$ (the flow preserves compacts) and $(\Phi_{H^0_{\mN}}^t)^*g_2\in \CR(T)$ (the flow of the simple harmonic oscillator preserves linear subspaces) we obtain
\begin{align}
\nonumber(\Phi_{H_{\mN}}^t)^*((g_1\circ\pi_{S})(g_2\circ\pi_{T}))&=
(((\Phi_{H_{\mN}}^t)^*g_1)\circ\pi_{S})(((\Phi_{H^0_{\mN}}^t)^*g_2)\circ\pi_{T})\\
&\in C_0(S)\hat{\otimes} \CR(T).
\end{align}
Hence \eqref{preserved} holds. The same holds for $(\Phi_{H_\mN}^{-t})^*$ implying that $(\Phi^t_{H_\mN})^*$ is an automorphism for the ideal.

To continue our proof we observe that, by definition of $\tilde U_\infty$ and $S$ and $T$, we may write $\tilde U_\infty=\pi_S^{-1}(U)$ for a precompact set $U\subset S$. 
By Urysohn's lemma, we may choose a continuous bump function $0\leq \tilde g\leq 1$ on $\Omega_\Lambda$ that is 1 on $\tilde U_\infty$ and supported in an slightly larger set containing $\tilde U_\infty$. Namely, we define $\tilde g:=\check g\circ \pi_S$ for a compactly supported $\check g\in C_\c(S)$ that is 1 on $U$. We then have $\tilde g\in C_0(S)\hat{\otimes} \CR(T)$. We then define $f_\infty:=\tilde{g}\cdot(\Phi_{H_\mN}^{t})^*g$. As $\pi_S(\tilde U_\infty)=U$, we obtain $f_\infty(\omega)=g\circ \Phi_{H_\mN}^{t}(\omega)$ for all $\omega\in \tilde U_\infty$. Moreover, 
$$f_\infty=\tilde g\cdot (g\circ \Phi_{H_\mN}^{t})=((\tilde g\circ \Phi_{H_\mN}^{-t})\cdot g)\circ\Phi_{H_\mN}^{t}.$$
By \eqref{preserved}, we have $\tilde g\circ \Phi_{H_\mN}^{-t}\in C_0(S)\hat{\otimes} \CR(T)$. 
Since $C_0(S)\hat{\otimes} \CR(T)\subseteq \CR(\Omega_\Lambda)$ is an ideal we obtain that $(\tilde g\circ \Phi_{H_\mN}^{-t})\cdot g\in C_0(S)\hat{\otimes} \CR(T)$. Again on account of \eqref{preserved} we obtain that $((\tilde g\circ \Phi_{H_\mN}^{-t})\cdot g)\circ\Phi^t_{H_\mN}\in C_0(S)\hat{\otimes} \CR(T)$. We conclude that $f_\infty\in \CR(\Omega_\Lambda)$, which finishes the proof.
\end{proof}

We are finally in a position to prove the following technical version of Proposition \ref{prop: compact case}.
\begin{proposition}\label{prop: compact case2}
Let $\Lambda\Subset\Gamma$, $\mN$ of the form \eqref{eq:mN}, and assume that $V_{kl}\in C_\c^\infty(\R^d)$ for all $k,l\in\mN$. Then, for all $t\in[0,1]$,
    \begin{align}\label{eq:to prove t_final3}
(\Phi_{H_\mN}^{t})^*(\CR(\Omega_\Lambda))\subseteq \CR(\Omega_\Lambda).
\end{align}
\end{proposition}
\begin{proof}

The proof of \eqref{eq:to prove t_final3} is performed by strong induction on the (finite) size of $\mathcal{N}$ appearing in $V_\mN=\sum_{\{k,l\}\in\mathcal{N}}2V_{kl}\circ \pi_{kl}$.  
If $|\mathcal{N}|=0$, the potential $V_\mN$ is zero and we obtain the simple harmonic oscillator for which we already know the statement holds (cf. Lemma \ref{lemm: simple time evolution}).
We now carry out the induction step. Assume the time evolution
with respect to $H_{\mN-kl}$ preserves $\CR(\Omega_\Lambda)$, for each $\{k,l\}\in\mathcal{N}$. We must show that the time evolution
with respect to $H_{\mN}$ preserves $\CR(\Omega_\Lambda)$. For a given $g\in \CR(\Omega_\Lambda)$ we write $f_{kl}:=(\Phi_{H_{\mN-kl}}^{t})^*g$ so that $f_{kl}\in \CR(\Omega_\Lambda)$. Fixing $f_\infty$ as in Lemma \ref{Lemm: existence function}, equation \eqref{uniform continuity} together with Lemma \ref{Lemm: existence function} implies that
\begin{align}\label{restriction identity}
    ||(\Phi_{H_{\mN}}^{t})^*g|_{U_Y}-f_Y|_{U_Y}||_\infty<\epsilon,
\end{align}
for each $Y\in\mathcal{N}\cup\{\infty\}$. We now construct a partition of unity $\{\eta_Y\}_{Y\in\mathcal{N}\cup\{\infty\}}$ subordinate to the cover $\{U_Y\}_{Y\in\mathcal{N}\cup\{\infty\}}$ of $\Omega_\Lambda$, to patch together the functions $f_Y$ and obtain a single function in $\CR(\Omega_\Lambda)$.
We start by defining non-negative smooth functions $\zeta_Y\in \CR(\Omega_\Lambda)$ that are $1$ on $W_Y$ and $0$ outside of $U_Y$. Namely, we take $\zeta_\infty:=g_\infty\circ \pi_S$ for some bump function $g_\infty$ on $S$, where $S$ defined in the previous lemma, and for each $Y=\{k,l\}\in\mathcal{N}$, we take $\zeta_{kl}$ to be either $g_{kl}\circ \pi^{\mom,\pos}_{kl}$ or $g_{kl}\circ\pi_{\{k,l\}}$ for some bump function $g_{kl}$ supported on the complement of a ball with radius  $D_{kl}$ (cf. \eqref{eq:Ukl}-\eqref{eq:Wkl}). Since $\{W_Y\}_{Y\in\mathcal{N}\cup\{\infty\}}$ is a cover of $\Omega_\Lambda$, the sum $\sum_Y \zeta_Y$ is bounded below by $1$ and therefore it is invertible in $\CR(\Omega_\Lambda)$. It follows that for each $Y\in\mathcal{N}\cup\{\infty\}$ the function
$$\eta_Y=\frac{\zeta_Y}{\sum_{Y'\in\mathcal{N}\cup\{\infty\}} \zeta_{Y'}}$$
lies in $\CR(\Omega_\Lambda)$. Moreover, \eqref{restriction identity} implies
\begin{align}
    ||(\Phi_{H_\mN}^{t})^*g - \sum_{Y\in\mathcal{N}\cup\{\infty\}}\eta_Yf_Y||_\infty<\epsilon.
\end{align}
Since $\epsilon>0$ was arbitrary and $\CR(\Omega_\Lambda)$ is norm-closed, the assertion follows.
\end{proof}

\begin{proof}[Proof of Prop. \ref{prop: compact case}]
By choosing $\mathcal{N}=\{\{k,l\}:k,l\in\Lambda,k\neq l\}$, so that $H_\mathcal{N}+c=H_\Lambda=H_\Lambda^0+V_\Lambda$, for a constant $c=\sum_{k\in\Lambda}V_{kk}(0)$, the first assertion of Proposition \ref{prop: compact case} follows directly from Proposition \ref{prop: compact case2}. The second assertion follows by the argument noted after the statement of Proposition \ref{prop: compact case}.
\end{proof}

\noindent
\subsection{Time evolution for general potentials}
We extend the result of the previous subsection to the general case ($V_{kl}\in C_0^1(\bR^d)$ with $\nabla V_{kl}$ Lipschitz). The following lemma provides the required approximation of generic potentials by compactly supported ones.

\begin{lemma}\label{lemm: approximation}
For $k,l\in\Lambda$, let $V_{kl}\in C^{1}_0(\mathbb{R}^d,\R)$ be such that $\nabla V_{kl}$ is Lipschitz. 
Then there exists a sequence $(V_{kl,m})_{m=1}^\infty\subseteq C_\c^\infty(\mathbb{R}^d,\R)$ such that $(\nabla V_{kl,m})_{m=1}^\infty$ converges uniformly to $\nabla V_{kl}$, as $m\to\infty$.
\end{lemma}
\begin{proof}
The fundamental theorem of calculus implies that $\nabla V_{kl}$ is necessarily bounded.
Since $C_\c^\infty(\R^d,\R)\subset C_0(\R^d,\R)$ is dense, we can extract a sequence $(\widetilde V_{kl,m})_{m=1}^\infty\subset C_\c^\infty(\R^d,\R)$ such that
     \begin{align}\label{approx2}
          ||\widetilde V_{kl,m} - V_{kl}||_\infty<\frac{1}{m^2}.
     \end{align}
     We introduce an approximation to the identity, i.e., we fix a positive function $h\in C_\c^\infty(\R^d)$ with $\int h(x)\,dx=1$ and let $h_m(x):=m^d h(mx)$ for all $m\in \mathbb{Z}_{\geq 0}$. 
     We define
 \begin{align*}
     V_{kl,m}:=h_m\ast \tilde V_{kl,m}.
 \end{align*}
 Note that  each $V_{kl,m}$ is compactly supported. 
 By a property of convolutions, 
 \begin{align*}
     \partial_i (V_{kl,m})=\partial_i h_m\ast \widetilde V_{kl,m},\qquad\partial_i (h_m\ast V_{kl})=h_m\ast \partial_iV_{kl},
 \end{align*}
 where $\partial_i$ is the $i^\text{th}$ partial derivative on $\R^d$.
 We compute
 \begin{align*}
     \|\partial_iV_{kl,m}-\partial_i V_{kl}\|_\infty\leq  \|\partial_iV_{kl,m}-\partial_ih_m\ast V_{kl}\|_\infty + \|\partial_i(h_m\ast V_{kl})-\partial_iV_{kl}\|_\infty\nonumber\\ \leq 
     \int_{\R^d} \partial_ih_m(x)\,dx\,||\widetilde V_{kl,m}-V_{kl}||_\infty+\|h_m\ast \partial_iV_{kl}-\partial_iV_{kl}\|_\infty.
 \end{align*}
 The first term on the right-hand side converges as a result of \eqref{approx2}, and the fact that $\int \partial_i h_m=m\int \partial_ih$ for all $m$. The second term converges since $h_m$ is an approximation to the identity, and $\partial_iV_{kl}$ is bounded.
 We therefore obtain the lemma.
\end{proof}

We now extend Proposition \ref{prop: compact case} to general $V$, thereby arriving at our final result for finite systems:
\begin{theorem}\label{thm: main}
Let $\Lambda\Subset\Gamma$ and let $V_{kl}\in C_0^{1}(\R^d,\R)$ ($k,l\in\Lambda$) be such that $V_{lk}(x)=V_{kl}(-x)$, and such that $\nabla V_{kl}$ is Lipschitz continuous (as is the case when Assumption \ref{conditions} is satisfied). Consider $H_\Lambda$ defined by \eqref{Hamiltoniannew2} with the ensuing flow $\Phi_{H_\Lambda}^t$, cf. \eqref{hamilton equations}-\eqref{hamilton equations8}. Then it holds
\begin{align}
    (\Phi_{H_\Lambda}^t)^*\CR(\Omega_\Lambda)=\CR(\Omega_\Lambda),
\end{align}
for all $t\in\mathbb{R}$.
\end{theorem}

\begin{proof}
By Lemma \ref{lemm: approximation}, we obtain a sequence of Hamiltonians 
$$H_{\Lambda,m}=H^0_\Lambda+\sum_{k,l\in\Lambda}V_{kl,m}\circ\pi_{kl}$$
whose corresponding dynamics $(\Phi_{H_{\Lambda,m}}^{t})^*$ preserve $\CR(\Omega_\Lambda)$ by Proposition \ref{prop: compact case}. (This was the point of the previous subsection.) The fact that the uniform convergence of gradients of Lemma \ref{lemm: approximation} implies that $(\Phi_{H_\Lambda}^{t})^*$ also preserves $\CR(\Omega_\Lambda)$ is proven exactly as is \cite[Theorem 15]{vanNuland_Stienstra_2019} (cf. \cite[Theorem 4.4.8]{vN}).
\end{proof}

\section{Infinite dimensional dynamical systems}\label{Section: infinite dynamics}
The main theorem of this section shows that the thermodynamic limit of the local dynamics $\alpha_{\Lambda}^t:\pi_\Lambda^*\CR(\Omega_\Lambda)\to\pi_\Lambda^*\CR(\Omega_\Lambda)$ defined by (cf. Def. \ref{def:time})
\begin{align}\label{eq:def alpha_Lambda}
\alpha_{\Lambda}^t(f\circ\pi_\Lambda):=f\circ\Phi^t_{H_\Lambda}\circ\pi_\Lambda\qquad(f\in\CR(\Omega_\Lambda))
\end{align}
is well defined and induces a one-parameter group of *-isomorphisms $\alpha^t$ on the resolvent algebra $\CR(\Omega)=\overline{\cup_{\Lambda\Subset\Gamma}\pi_\Lambda^*\CR(\Omega_\Lambda)}$. 
\begin{theorem}\label{thm: CR stable}
Given a countable set $\Gamma$, constants $m_k,\nu_k>0$, and functions $V_{kl}:\R^d\to\R$ for all $k,l\in\Gamma$ satisfying Assumption \eqref{conditions}, let $\alpha^t_\Lambda$ be the time flow defined by \eqref{eq:def alpha_Lambda}, for all $t\in\R$ and $\Lambda\Subset\Gamma$. Then, for all $t\in\R$, there exists a (unique) *-isomorphism
$$\alpha^t:\CR(\Omega)\to\CR(\Omega),$$
satisfying
\begin{align}\label{eq:def alpha^t}
    \alpha^t(f)=\lim_{\Lambda\nearrow\Gamma}\alpha_\Lambda^t(f)
\end{align}
for all $\Lambda_0\Subset\Gamma$ and $f\in\pi_{\Lambda_0}^*\CR(\Omega_{\Lambda_0})$. Note that \eqref{eq:def alpha^t} determines $\alpha^t$ on a dense subset of $\CR(\Omega)$, hence defines $\alpha^t$ uniquely.
Moreover, this collection of *-isomorphisms $(\alpha^t)_{t\in\R}$ is a one-parameter group. 
\end{theorem}
We spend the rest of this section to prove the above theorem.
Part of the approach is inspired by similar works in the context of quantum mechanics, e.g. \cite[Section 3]{Robinson}, \cite[Sections 7.1 and 10]{Buchholz_Grundling_2008} and \cite[Section 3.3.2]{Naaijkens_2017}. The two novelties here are firstly that the same methods work in setting of classical Hamiltonian dynamics, and secondly that assumptions on the material structure can be relaxed.

In fact, our approach does not require us to assume the existence of a metric on $\Gamma$, nor does it require us to introduce some type of $\Gamma$-regularity  or a ``Lieb-Robinson'' type norm, as done in e.g. \cite{KootVen}.
Besides the obvious benefits of simple assumptions, it provides a step towards the inclusion of more abstract models such as those arising in noncommutative geometry or the continuum limit of lattice gauge theory.
\\\\
We use the following notation. For every finitely supported $\beta:\Gamma\times\{1,\cdots 2d\}\to\mathbb{Z}_{\geq 0}$, with $|\beta|=\sum_{k\in\Gamma}\sum_{i=1}^{2d}\beta(k,i)$, we write
\begin{align}
    \partial^\beta=\prod_{k\in\Gamma}\prod_{1=1}^{2d}\partial_{k,i}^{\beta(k,i)},
\end{align}
where $\partial_{k,i}^{\beta(k,i)}=\partial^{\beta(k,i)}/\partial_{k,i}$, the usual partial derivative of order $\beta(k,i)$ acting on coordinate $i$ of site $k$.
The following lemma provides a time-independent bound on the non-interacting time evolution of potentials.
\begin{lemma}\label{lemm:estimate Poisson bracket}
Let $k,l\in\Lambda\Subset\Gamma$ and $t\in\R$, and define $V_{kl,t}:=V_{kl}\circ\pi_{kl}\circ\Phi_{H_{\Lambda}^0}^t$, 
i.e. the time evolution of $V_{kl}\circ\pi_{kl}$ with respect to simple harmonic motion. Then, for any $\beta:\Lambda\times\{1,\ldots,2d\}\to\Z_{\geq0}$
\begin{align}
    &\|\partial^\beta{V_{kl,t}}\|_\infty\leq C_{kl}C_V^{|\beta|},
\end{align}
where $C_V>0$ is a constant independent of $\Lambda,k,l$ and $t$.
\end{lemma}
\begin{proof}
Without loss of generality we may assume that $\Lambda=\{k,l\}$, as simple harmonic flow splits completely. For any function $g:\R^N\to\R$ and any $N\times N$ matrix $S$ we have for all $n\in\mathbb{Z}_{\geq 0}$ and $i_1,\ldots,i_n\in\{1,\ldots,N\}$,
$$\partial_{i_1}\cdots\partial_{i_n}(g\circ S)=\sum_{j_1,\ldots,j_n=1}^N((\partial_{j_1}\cdots \partial_{j_n}g)\circ S)S_{j_1i_1}\cdots S_{j_ni_n}.$$
We note that $\Phi_{H^0_\Lambda}^t$ is a linear map on the $N=4d$-dimensional vector space $\Omega_{\Lambda}=\Omega_{\{k,l\}}$, and hence has an associated matrix. We let $S=\Phi_{H^0_\Lambda}^t$ be this matrix, we take $g=V_{kl}\circ\pi_{kl}$, and we apply Assumption \ref{conditions}(iii a) to find
$$\|\partial_{i_1}\cdots\partial_{i_n}V_{kl,t}\|_\infty\leq C_{kl}C^nc_d^n\|S\|^n\qquad(i_1,\ldots,i_n\in\Lambda\times\{1,\ldots,2d\}),$$
where $c_d$ depends on the specific matrix norm used in $\|S\|$ (they are all equivalent).

We only need to prove that the number $\|S\|$ is bounded independently of $k,l$ and $t$. Note that by definition of $H_{\Lambda}^0$, the flow $\Phi_{H_\Lambda^0}^t$ decomposes into $2d$ independent simple harmonic oscillators, each defined on $\R^2$, and given by
\begin{align*}
    & p_{a,i}^t=-\sqrt{\nu_am_a}q_{a,i}\sin(\sqrt{\frac{\nu_a}{m_a}}t)+p_{a,i}\cos(\sqrt{\frac{\nu_a}{m_a}} t) \  &(a\in\{k,l\}; \ i=1,\cdots, d); \\
    &q_{a,i}^t=q_{a,i}\cos(\sqrt{\frac{\nu_a}{m_a}}t)+\frac{p_{a,i}}{\sqrt{\nu_am_a}}\sin(\sqrt{\frac{\nu_a}{m_a}} t) \  &(a\in\{k,l\}; \ i=1,\cdots,d),
\end{align*}
where $(p^t_{a,i},q^t_{a,i})$ denotes the $i^{th}$-component of $\Phi^t_{H_\Lambda^0}(p,q)$ in direction $a$. 
The above expressions imply that the matrix $S$ associated with the flow $\Phi_{H_\Lambda^0}^t$  in the standard $(p,q)$ basis is independent of $q$ and $p$; all its entries are bounded by a cosine or a sine times  $\sup_{a\in\Gamma}(\sqrt{\nu_am_a},1/\sqrt{\nu_am_a})$, which is a finite number on account of Assumption \ref{conditions}(iv).  We conclude that  $C_V=Cc_d\|S\|$ is bounded independently of $k,l$ (hence of $\Lambda=\{k,l\}$) and $t$.
\end{proof}

For any observable $f\in\pi_{\Lambda_0}^*\CR(\Omega_{\Lambda_0})$,
and any $\Lambda_0\subseteq\Lambda\Subset\Gamma$ we consider 
\begin{align}\label{expansion integral gamma}
    \gamma_\Lambda^{t}(f):=
    f\circ (\Phi_{H_\Lambda^0}^{-t}\circ\Phi_{H_\Lambda}^{t})\circ\pi_\Lambda,
\end{align}
$t>0$. Let us now fix an arbitrary $g\in\SR(\Omega_\Lambda)$. Lemma \ref{lemm: simple time evolution} implies that $g\circ\Phi^t_{H_\Lambda^0}\in\SR(\Omega_\Lambda).$ The main result of Section \ref{Sect:finite cases}, Theorem \ref{thm: main} entails that, $\gamma_\Lambda^{t}(g\circ\pi_\Lambda)\in \CR(\Omega)$.

Using $\frac{d}{ds}f\circ\Phi_{H}^s=\{f,H\}\circ \Phi_H^s$ (we view $H=H_\Lambda,H_{\Lambda}^0,\Phi_{H}^s$, $V_\Lambda$ as functions on $\Omega$ by composing with $\pi_\Lambda$), as well as $\{f\circ\Phi^{-s}_{H_\Lambda^0},V_\Lambda\}=\{f,V_\Lambda\circ\Phi^s_{H^0_\Lambda}\}\circ\Phi^{-s}_{H_\Lambda^0}$, and \eqref{expansion integral gamma} gives the following identity
\begin{align*}
   \gamma_\Lambda^t(f)&=f+\int_0^t\frac{d}{ds}(f\circ\Phi_{H_\Lambda^0}^{-s}\circ\Phi_{H_\Lambda}^s)ds\\
   &=f+\int_0^t\Big(\{f\circ\Phi_{H_\Lambda^0}^{-s},H_\Lambda\}\circ\Phi_{H_\Lambda}^{s}-\{f,H^0_\Lambda\}\circ\Phi_{H_\Lambda^0}^{-s}\circ\Phi_{H_\Lambda}^{s}\Big)ds\\
   &=f+\int_0^t\{f\circ\Phi_{H_\Lambda^0}^{-s},V_\Lambda\}\circ\Phi_{H_\Lambda}^{s}ds\\
   &=
   f+\int_0^t\gamma_\Lambda^s(\{f,V_\Lambda\circ\Phi_{H_\Lambda^0}^{s}\})ds\\
   &=f+\int_0^t\gamma_\Lambda^s(\{f,V_{\Lambda,s}\})ds,
\end{align*}
where $V_{\Lambda,s}:=V_\Lambda\circ\Phi_{H_\Lambda^0}^{s}$.
Iterating the right-hand side of the above equation yields a Dyson series with remainder
\begin{align}\label{eq:Dysonreeks met remainder}
     \nonumber&\gamma_\Lambda^t(f)=f+\\
     \nonumber&\sum_{n=1}^{M-1}\int_0^tds_n\int_0^{s_n}ds_{n-1}\cdots\int_0^{s_2}ds_1\{\{\{f,V_{\Lambda,s_1}\},V_{\Lambda,s_2}\}\cdots\},V_{\Lambda,s_n})\}+\\
     &\int_0^tds_M\int_0^{s_M}ds_{M-1}\cdots\int_0^{s_2}ds_1\gamma^{s_M}_\Lambda(\{\{\{f,V_{\Lambda,s_1}\},V_{\Lambda,s_2})\}\cdots\},V_{\Lambda,s_n}\}).
\end{align}
The goal is to estimate the integrand. We first point out that it is not clear \textit{a priori} whether the remainder vanishes as $M\to\infty$, as it involves derivatives of $f$ as $V_{kl}$ of increasing order in $M$. To solve this issue we consider the sets
\begin{align*}
\DhR(\Omega_\Lambda):=\{g\circ P \ |\ \hat{g}\in C_c^\infty(\ran P),P:\Omega_\Lambda\to\Omega_\Lambda\text{ projection}\},
\end{align*}
where $\hat{\cdot}$ denotes the Fourier transform. This set is dense in the space $\mathcal{S}_\mathcal{R}(\Omega_\Lambda)$ which in turn is dense in $\CR(\Omega_\Lambda)$. 
Given $f=g\circ P\in\DhR(\Omega_\Lambda)$, the number $B=2\pi\sup_{\hat{g}(x)\neq0}\|x\|$ is finite. It follows that for any multi-index $\beta_0:\Lambda\times\{1,\ldots,2d\}\to\Z_{\geq0}$ the higher-order derivative $\partial^{\beta_0}f$ satisfies
\begin{align}\label{eq:partial f}
    \|\partial^{\beta_0} f\|_\infty\leq (2\pi)^{|\beta_0|}\left\|\|x\|^{|\beta_0|}\hat{g}(x)\right\|_{L^1}\leq B^{|\beta_0|}\|\hat g\|_{L^1}\leq C_f^{|\beta_0|},
\end{align}
for a number $C_f=\max(B,\|\hat{g}\|_{L_1},1)$ only depending on $f$.
In the following discussion we take a fixed but arbitrary $f\in \pi_{\Lambda_0}^*\DhR(\Omega_{\Lambda_0}),$ localized on a fixed $\Lambda_0\Subset\Gamma$. The nested Poisson brackets appearing in \eqref{eq:Dysonreeks met remainder} have a larger support than $f$. The idea is to  accurately estimate this support as a function on the number of iterations, keeping in mind that at zero order, this is just the support of $f$, i.e. $\Lambda_0$.
To make this precise we proceed in the following manner.



Let $n\in\mathbb{Z}_{\geq 0}$ and let $\boldsymbol{k}=(k_1,\ldots,k_n),\boldsymbol{l}=(l_1,\ldots,l_n)\in\Gamma^n$ satisfy $k_j\neq l_j$ for every $j=1,\ldots,n$. Below we shall define a set $B_{\boldsymbol{kl}}^n$ of tuples $\boldsymbol\beta=(\beta_0,\ldots,\beta_n)$ of finitely supported multi-indices $\beta_j:\Gamma\times\{1,\ldots,2d\}\to\Z_{\geq0}$ such that 
\begin{align}\label{eq:Bkln}
\{\cdots\{\{f_{\Lambda_0},f_{k_1l_1}\},f_{k_2l_2}\}\cdots,f_{k_nl_n}\}=\sum_{\boldsymbol\beta\in B_{\boldsymbol{kl}}^n}\epsilon_{\boldsymbol\beta}n_{\boldsymbol{\beta}}\partial^{\beta_0}f_{\Lambda_0}\partial^{\beta_1}f_{k_1l_1}\cdots\partial^{\beta_n}f_{k_nl_n},
\end{align}
for all functions $f_{\Lambda_0}\in\pi_{\Lambda_0}^*\CR(\Omega_{\Lambda_0})$ and $f_{k_jl_j}\in\pi_{\{k_j,l_j\}}^*\CR(\Omega_{\{k,l\}})$ (the subscripts indicating on which finite subspace of $\Omega$ the functions depend),
where $\epsilon_{\boldsymbol\beta}\in\{\pm 1\}$ takes into account the antisymmetry of the Poisson bracket, and $n_{\boldsymbol{\beta}}\in\N$ is a combinatorial coefficient. 
Indeed, we may define $B_{\boldsymbol{kl}}^n$ recursively, by (using the standard Kronecker multi-indices $\delta_{(a,i)}\in\Z_{\geq0}^{\Gamma\times\{1,\ldots,2d\}}$)
\begin{align*}
    B_{\boldsymbol{kl}}^1:=\Big\{&\boldsymbol\beta=(\delta_{(a,i)},\delta_{(a,i+d)})
    ~:~ a\in\Lambda_0\cap\{k_1,l_1\}, i\in\{1,\ldots,2d\}\Big\}\\
    B_{\boldsymbol{kl}}^{n+1}:=\Big\{&\boldsymbol\beta=(\beta_0,\ldots,\beta_{n},\delta_{(a,i+d)})~:~i\in\{1,\ldots,2d\},\\
    &a\in(\Lambda_0\cup\{k_1,\ldots,k_{n},l_1,\ldots,l_n\})\cap\{k_{n+1},l_{n+1}\},\\
    &\exists j\in\{1,\ldots,n\}:(\beta_0,\ldots,\beta_j-\delta_{(a,i)},\ldots,\beta_{n})\in B_{\boldsymbol{kl}}^{n}\Big\},
\end{align*}
where $i+d$ is interpreted modulo $2d$. 
Note that each $\boldsymbol{\beta}\in B^n_{\boldsymbol{kl}}$ satisfies the conditions $|\beta_0|+\ldots+|\beta_n|=2n$, $\beta_0(a)=0$ for $a\notin\Lambda_0$, and $\beta_j(a)=0$ for $a\notin\{k_j,l_j\}$. These nested Poisson brackets expand into a binomial-type sum with two choices at each level, determined by the $\pm$ structure of the Poisson bracket.
\begin{remark}\label{rem:hypergraph}
Let $G_{\boldsymbol{kl}}^n$ be the hypergraph whose vertices are the elements of $\{k_1,l_1,\ldots,k_n,l_n\}\cup\Lambda_0$ and whose hyperedges are the sets $\{k_1,l_1\},\ldots,\{k_n,l_n\}$ and $\Lambda_0$. Then $B_{kl}^n$ is nonempty only if $G_{\boldsymbol{kl}}^n$ is connected. In fact $B_{kl}^n\neq\emptyset$ if and only if $G^{i}_{\boldsymbol{kl}}$ is connected for all $i=1,\ldots,n$. As $k_j\neq l_j$, $G_{\boldsymbol{kl}}^n$ does not contain self-loops. Note that $G_{\boldsymbol{kl}}^n$ is almost a graph, the only exception is the hyperedge $\Lambda_0$.
\end{remark}
From the sets $B_{\boldsymbol{kl}}^n$ we extract the following numbers
\begin{align}\label{def Rkl}
 R_{\boldsymbol{kl}}^n:=\sum_{\boldsymbol\beta\in B_{\boldsymbol{kl}}^n}n_{\boldsymbol{\beta}}C_f^{|\beta_0|},
\end{align}
where $C_f\geq1$ is defined by \eqref{eq:partial f}.
The numbers $R_{\boldsymbol{kl}}^n$ satisfy the following recursive relation. 
\begin{lemma}\label{graph lemma}
Let $\Lambda_0\Subset\Gamma$, $\boldsymbol{k},\boldsymbol{l}\in\Gamma^n$ and $f\in \pi_{\Lambda_0}^*\DhR(\Omega_{\Lambda_0})$. 
Let $G^n_{\boldsymbol{kl}}$ be the hypergraph of Remark \ref{rem:hypergraph}.
Assume that the graph is connected and does not contain self-loops.
The following relation holds
\begin{align}\label{relation}
    R_{\boldsymbol{kl}}^n=2d({\rm wdeg}_{G_{\boldsymbol{kl}}^{n-1}}(k_n)+{\rm wdeg}_{G_{\boldsymbol{kl}}^{n-1}}(l_n))R_{\boldsymbol{kl}}^{n-1},
\end{align}
where the numbers $R_{\boldsymbol{kl}}^n$ are given by \eqref{def Rkl} (and $R_{\boldsymbol{kl}}^0=1$), and for every $a\in\Gamma$ we define the number
\begin{align}
    {\rm wdeg}_{G_{\boldsymbol{kl}}^{n-1}}(a)=C_f1_{\Lambda_0}(a)+\sum_{j=1}^{n-1} 1_{\{k_j,l_j\}}(a),
\end{align}
which we think of as a ``weighted degree'' of the vertex $a$ in the graph $G_{\boldsymbol{kl}}^{n-1}=G^{n-1}_{(k_1,\ldots,k_{n-1},l_1,\ldots,l_{n-1})}$.
\end{lemma}
\begin{remark}\label{Rmk:relation}
The weighted degree of any vertex $a\in\Gamma$ is almost the regular degree, except that the hyperedge $\Lambda_0$ contributes $C_f$ instead of $1$ to the degree. In particular, for $a\notin\Lambda_0$, the corresponding weighted degree is the usual degree of the graph, and ${\rm wdeg}_{G_{\boldsymbol{kl}}^{n}}(a)=0$ if $a$ is not a vertex of $G^n_{\boldsymbol{kl}}$.
\end{remark}

\begin{proof}[Proof of Lemma \ref{graph lemma}]
Let us first check the case $n=1$. If both $k_1\in \Lambda_0$ and $l_1\in \Lambda_0$, then ${\rm wdeg}_{G_{\boldsymbol{kl}}^0}(k_1)=C_f={\rm wdeg}_{G_{\boldsymbol{kl}}^0}(l_1)$, and now $R_{\boldsymbol{kl}}^1=2(2d)C_f$, which follows from the assumption $k_1\neq l_1$ and the choices $\boldsymbol{\beta}=(\beta_0,\beta_1)=(\delta_{(k_1,i)},\delta_{(k_1,i)})$ and $\boldsymbol{\beta}=(\delta_{(l_1,i)},\delta_{(l_1,i)})$ for $i=1,\ldots,2d$. Therefore, both sides of \eqref{relation} equal $(2d)2C_f$. If $k_1\in\Lambda_0$ and $l_1\notin\Lambda_0$ then $R_{\boldsymbol{kl}}^1=(2d)C_f$, and these vertices have weighted degrees  ${\rm wdeg}_{G_{\boldsymbol{kl}}^0}(k_1)=C_f$ and  ${\rm wdeg}_{G_{\boldsymbol{kl}}^0}(l_1)=0$. Therefore, \eqref{relation} equals $(2d)C_f$. The same argument applies if $k_1\notin\Lambda_0$ and $l_1\in\Lambda_0$.
\\
To prove the general case we compute 
\begin{align*}
&\{\cdots\{\{f_{\Lambda_0},f_{k_1l_1}\},f_{k_2l_2}\}\cdots,f_{k_nl_n}\}\\
&=\sum_{\boldsymbol\beta\in B_{\boldsymbol{kl}}^{n-1}}\epsilon_{(\beta_0,\ldots,\beta_{n-1})}n_{(\beta_0,\ldots,\beta_{n-1})}\{\partial^{\beta_0}f_{\Lambda_0}\partial^{\beta_1}f_{k_1l_1}\cdots\partial^{\beta_{n-1}}f_{k_{n-1}l_{n-1}},f_{k_nl_n}\}\\
&=\sum_{\boldsymbol\beta\in B_{\boldsymbol{kl}}^{n-1}}\epsilon_{(\beta_0,\ldots,\beta_{n-1})}n_{(\beta_0,\ldots,\beta_{n-1})}\sum_{i=1}^{2d}\epsilon_{i\leq d}\Bigg(\\
&\quad\partial_{k_n,i}\Big(\partial^{\beta_0}f_{\Lambda_0}\partial^{\beta_1}f_{k_1l_1}\cdots\partial^{\beta_{n-1}}f_{k_{n-1}l_{n-1}}\Big)\partial_{k_n,i+d}f_{k_nl_n}\\
&\quad+\partial_{l_n,i}\Big(\partial^{\beta_0}f_{\Lambda_0}\partial^{\beta_1}f_{k_1l_1}\cdots\partial^{\beta_{n-1}}f_{k_{n-1}l_{n-1}}\Big)\partial_{l_n,i+d}f_{k_nl_n}\Bigg),
\end{align*}
where $\epsilon_{i\leq d}=-1$ $(i\leq d)$, $\epsilon_{i\leq d}=1$ $(i>d)$ represents the sign of the Poisson bracket.
We can then apply the Leibniz rule and note that $$\partial_{a,i}\partial^{\beta_j} f_{k_jl_j}=1_{\{k_j,l_j\}}(a)\partial^{\beta_j+\delta_{(a,i)}}f_{k_jl_j}$$ and similarly for $f_{\Lambda_0}$. By comparing the resulting formula with \eqref{eq:Bkln}, which defines $B_{\boldsymbol{kl}}^n$, one straightforwardly arrives at the lemma.
\end{proof}

We will now use the Dyson series \eqref{eq:Dysonreeks met remainder} to show that the small-time time-evolutions of elements of $\DhR(\Omega_{\Lambda_0})$ converge in the thermodynamic limit, thus making the main technical step towards Theorem \ref{thm: main}.

\begin{proposition}\label{prop:dynamics on Dhat}
Given a countable set $\Gamma$, constants $m_k,\nu_k>0$, and functions $V_{kl}:\R^d\to\R$ for all $k,l\in\Gamma$ satisfying Assumption \eqref{conditions}, let $\alpha^t_\Lambda$ be the time flow corresponding to $H_\Lambda$ and defined by Definition \ref{def:time}, for all $t\in\R$ and $\Lambda\Subset\Gamma$. There exists a constant $t_\lambda>0$ such that for all $\Lambda_0\Subset\Gamma$ and all $f\in\pi_{\Lambda_0}^*\DhR(\Omega_{\Lambda_0})$ we have (assuming $\Lambda_0\subseteq\Lambda,\Lambda'$)
$$\lim_{\Lambda,\Lambda'\nearrow\Gamma}\|\alpha_\Lambda^t(f)-\alpha_{\Lambda'}^t(f)\|_\infty=0,$$
for all $|t|< t_\lambda$.
\end{proposition}
\begin{proof}
To demonstrate the thesis we use the regularity condition (iii) from Assumption \eqref{conditions} as follows. For each $t\in\R$ and $k,l\in\Lambda\Subset\Gamma$ we denote $V_{kl,t}:=V_{kl}\circ\pi_{kl}\circ\Phi^t_{H_{\Lambda}^0}\circ\pi_{\Lambda}\in\CR(\Omega)$, so that $V_{\Lambda,t}:=\sum_{k,l\in\Lambda}V_{kl,t}=V_\Lambda\circ\Phi^t_{H^0_\Lambda}\circ\pi_\Lambda\in\CR(\Omega)$. Note that $V_{kl,t}$ is actually independent of $\Lambda$, and coincides with $V_{kl,t}$ of Lemma \ref{lemm:estimate Poisson bracket} interpreted as function on $\Omega$ using the injective *-homomorphism $\pi_\Lambda^*:\CR(\Omega_\Lambda)\to\CR(\Omega)$. We recall from \eqref{eq:partial f} and Lemma \ref{lemm:estimate Poisson bracket} that there exist constants $C_V,C_f\geq 1$ such that for any finitely supported $\beta:\Gamma\times\{1,\ldots,2d\}\to\Z_{\geq0}$, any $k,l\in\Gamma$, and any $t\in\R$, we have
\begin{align}\label{eq:bound partial V}
\|\partial^\beta f\|_\infty\leq C_f^{|\beta|},\qquad\|\partial^\beta V_{kl,t}\|_\infty
\leq C_{kl}C_V^{|\beta|}.
\end{align}

These ingredients allow us to bound the terms and remainder of the Dyson series \eqref{eq:Dysonreeks met remainder}. To this end, we first observe
\begin{align*}
&\|\{\cdots\{\{f,V_{\Lambda,s_1}\},V_{\Lambda,s_2}\}\cdots,V_{\Lambda,s_n}\}\|_\infty\\
&\leq\sum_{\substack{k_1,\ldots,k_n\\l_1,\ldots,l_n}}\|\{\cdots\{\{f,V_{k_1l_1,s_1}\},V_{k_2l_2,s_2}\}\cdots,V_{k_nl_n,s_n}\}\|_\infty\\
&\leq\sum_{\substack{k_1,\ldots,k_n\\l_1,\ldots,l_n}}\sum_{\boldsymbol\beta\in B_{\boldsymbol{kl}}^n}|\epsilon_{\beta}|n_{\boldsymbol{\beta}}\|\partial^{\beta_0}f\|_\infty\|\partial^{\beta_1}V_{k_1l_1,s_1}\|_\infty\cdots\|\partial^{\beta_n}V_{k_nl_n,s_n}\|_\infty\\
&\leq\sum_{\substack{k_1,\ldots,k_n\\l_1,\ldots,l_n}}\sum_{\boldsymbol\beta\in B_{\boldsymbol{kl}}^n}n_{\boldsymbol{\beta}}C_f^{|\beta_0|}C_{k_1l_1}\cdots C_{k_nl_n}C_V^{|\beta_1|+\ldots+|\beta_n|}\\
&\leq C_V^{2n}\sum_{\substack{k_1,\ldots,k_n\\l_1,\ldots,l_n}}R_{\boldsymbol{kl}}^nC_{k_1l_1}\cdots C_{k_nl_n},
\end{align*}
where the sum is over all indices $k_1,\ldots,k_n,l_1,\ldots,l_n$ such that $k_j\neq l_j$ and $\{k_j,l_j\}\cap(\Lambda_0\cup\{k_1,\ldots,k_{j-1},l_1,\ldots,l_{j-1}\})\neq\emptyset$ for all $j=1,\ldots,n$. We have used that for all $\boldsymbol\beta=(\beta_0,\ldots,\beta_n)\in B_{\boldsymbol{kl}}^n$ we in particular have $|\beta_0|+\ldots+|\beta_n|=2n$, so $|\beta_1|+\ldots+|\beta_n|\leq 2n$. 
Applying Lemma \ref{graph lemma} yields
\begin{align*}
&\|\{\cdots\{\{f,V_{\Lambda,s_1}\},V_{\Lambda,s_2}\}\cdots,V_{\Lambda,s_n}\}\|_\infty\\
&\leq C_V^{2n}\sum_{\substack{k_1,\ldots,k_{n}\\l_1,\ldots,l_{n}}}2d({\rm wdeg}_{G_{\boldsymbol{kl}}^{n-1}}(k_n)+{\rm wdeg}_{G_{\boldsymbol{kl}}^{n-1}}(l_n))R_{\boldsymbol{kl}}^{n-1} C_{k_1l_1}\cdots C_{k_nl_n}\\
&\leq 2dC_V^{2n}\sum_{\substack{k_1,\ldots,k_{n-1}\\l_1,\ldots,l_{n-1}}}R_{\boldsymbol{kl}}^{n-1}C_{k_1l_1}\cdots C_{k_{n-1}l_{n-1}}\\
&\quad\cdot\sum_{k_n,l_n\in\Gamma}({\rm wdeg}_{G_{\boldsymbol{kl}}^{n-1}}(k_n)+{\rm wdeg}_{G_{\boldsymbol{kl}}^{n-1}}(l_n)) C_{k_nl_n}\\
&= 2dC_V^{2n}\sum_{\substack{k_1,\ldots,k_{n-1}\\l_1,\ldots,l_{n-1}}}R_{\boldsymbol{kl}}^{n-1}C_{k_1l_1}\cdots C_{k_{n-1}l_{n-1}}\\
&\qquad\cdot\Big(\sum_{k_n\in\Gamma}{\rm wdeg}_{G_{\boldsymbol{kl}}^{n-1}}(k_n)\sum_{l_n\in\Gamma}C_{k_nl_n}+\sum_{l_n\in\Gamma}{\rm wdeg}_{G_{\boldsymbol{kl}}^{n-1}}(l_n)) \sum_{k_n\in\Gamma}C_{k_nl_n}\Big)\\
&\leq 2dC_V^{2n}\sum_{\substack{k_1,\ldots,k_{n-1}\\l_1,\ldots,l_{n-1}}}R_{\boldsymbol{kl}}^{n-1}C_{k_1l_1}\cdots C_{k_{n-1}l_{n-1}}\Big(2\sum_{k_n\in\Gamma}{\rm wdeg}_{G_{\boldsymbol{kl}}^{n-1}}(k_n)c\Big)\\
&= 4cdC_V^{2n}\sum_{\substack{k_1,\ldots,k_{n-1}\\l_1,\ldots,l_{n-1}}}R_{\boldsymbol{kl}}^{n-1}C_{k_1l_1}\cdots C_{k_{n-1}l_{n-1}}(C_f|\Lambda_0|+2n-2),
\end{align*}
where we have introduced a constant $c=\sup_k\sum_l C_{kl}=\sup_l\sum_k C_{kl}$ in the second to last line, which is finite by Assumption \ref{conditions}(iii b).
Repeating the above procedure yields
\begin{align*}
&\|\{\cdots\{\{f,V_{\Lambda,s_1}\},V_{\Lambda,s_2}\}\cdots,V_{\Lambda,s_n}\}\|_\infty\\
&\leq (4cd)^nC_V^{2n}(C_f|\Lambda_0|)(C_f|\Lambda_0|+2)\cdots (C_f|\Lambda_0|+2n-2)\\
&\leq (4cdC_V^2)^n(C_f|\Lambda_0|+2n-2)^n.
\end{align*}
Following Naaijkens \cite[Sect. 3]{Naaijkens_2017}, and  Robinson \cite{Robinson}, we set
\begin{align}
    a=C_f|\Lambda_0|+2n-2
\end{align}
and use that $a^n\leq n!\lambda^{-n}\exp(a\lambda)$ for all positive $a$ and $\lambda$, to obtain
\begin{align}
    (C_f|\Lambda_0|+2n-2)^n\leq n!\lambda^{-n}\exp((C_f|\Lambda_0|+2n-2)\lambda).
\end{align}
We thus find
\begin{align*}
\|\{\cdots\{\{f,V_{\Lambda,s_1}\},V_{\Lambda,s_2}\}\cdots,V_{\Lambda,s_n}\}\|_\infty\leq n!\exp((C_f|\Lambda_0|-2)\lambda)\bigg{(}\frac{4cdC_V^{2} e^{2\lambda}}{\lambda}\bigg{)}^n.
\end{align*}

By the above bounds, the $n^\text{th}$ term in the Dyson series \eqref{eq:Dysonreeks met remainder} can be bounded by
\begin{align*}
    \frac{|t|^n}{n!}\|\{\cdots\{\{f,V_{\Lambda,s_1}\},V_{\Lambda,s_2}\}\cdots,V_{\Lambda,s_n}\}\|_\infty&\leq \exp((C_f|\Lambda_0|-2)\lambda)\bigg{(}\frac{4cdC_V^{2}e^{2\lambda}}{\lambda}\bigg{)}^n|t|^n,
\end{align*}
noting that the $\frac{|t|^n}{n!}$ is the measure of the integral. Since $\sum_nx^n$ converges for $|x|<1$, the Dyson series converges whenever
    $$|t|< t_\lambda:=\bigg{(}\frac{4cdC_V^{2}e^{2\lambda}}{\lambda}\bigg{)}^{-1}.$$
The same bound holds for the remainder of \eqref{eq:Dysonreeks met remainder}, as $\gamma_\Lambda^t$ is isometric, which shows that for all $|t|<t_\lambda$ we have
\begin{align}\label{eq:Dysonreeks abs conv}
     &\gamma_\Lambda^t(f)=f+\\&\sum_{n=1}^{\infty}\int_0^tds_n\int_0^{s_n}ds_{n-1}\cdots\int_0^{s_2}ds_1\{\cdots\{\{f,V_{\Lambda,s_1}\},V_{\Lambda,s_2}\}\cdots,V_{\Lambda,s_n}\}\nonumber
\end{align}
and this Dyson series converges absolutely. To prove the dynamics extends globally we proceed as follows.
By the same arguments as above, for each $n\in\mathbb{Z}_{\geq 0}$ we have
\begin{align*}
     &\sum_{(\boldsymbol{k},\boldsymbol{l})\in\Gamma^{2n}}\supnorm{\{\cdots\{\{f,V_{k_1l_1,s_1}\},V_{k_2l_2,s_2}\}\cdots,V_{k_nl_n,s_n}\}}\\
     &\leq n!\exp((C_f|\Lambda_0|-2)\lambda)\frac{1}{t_\lambda^n}<\infty,
\end{align*}
whenever $0\leq |s_1|\leq\cdots\leq |s_n|\leq |t|< t_\lambda$.
This in particular shows that, although $\Gamma$ is infinite, the above series converges (absolutely). In fact,
\begin{align*}
&\sum_{n=0}^\infty \sum_{(\boldsymbol{k},\boldsymbol{l})\in\Gamma^{2n}}\int_0^tds_n\int_0^{s_n}ds_{n-1}\cdots\\
&\qquad\cdots\int_0^{s_2}ds_1\|\{\{\{f,V_{k_1l_1,s_1}\},V_{k_2l_2,s_2}\}\cdots,V_{k_nl_n,s_n}\}\|_\infty<\infty.
\end{align*}
Hence, for all $\epsilon>0$ there exists a finite partial sum of the above series, with index set $F\Subset \coprod_{n=0}^\infty \Gamma^{2n}=\{(n,(\boldsymbol{k},\boldsymbol{l})):n\geq0,\boldsymbol{k},\boldsymbol{l}\in\Gamma^n\}$, say, such that for all $|t|<t_\lambda$ we have
\begin{align*}
&\sum_{\substack{(n,(\boldsymbol{k},\boldsymbol{l}))\in(\coprod_{n=0}^\infty \Gamma^{2n})\setminus F}}\int_0^tds_n\int_0^{s_n}ds_{n-1}\cdots\\
&\qquad\cdots\int_0^{s_2}ds_1\|\{\{\{f,V_{k_1l_1,s_1}\},V_{k_2l_2,s_2}\}\cdots,V_{k_nl_n,s_n}\}\|_\infty<\epsilon.
\end{align*}
By defining
$$\Lambda_\epsilon:=\Lambda_0\cup \bigcup_{(n,(\boldsymbol{k},\boldsymbol{l}))\in F}\{k_1,l_1,\ldots,k_n,l_n\},$$
we find that $\Lambda_\epsilon\Subset\Gamma$ is finite and $F\subseteq \coprod_{n=0}^\infty\Lambda_\epsilon^{2n}$.

Let $\Lambda\supseteq\Lambda'\supseteq\Lambda_\epsilon$ be two finite subsets of $\Gamma$. As both Dyson series \eqref{eq:Dysonreeks abs conv} of $\gamma_\Lambda^t(f)$ and $\gamma_{\Lambda'}^t(f)$ absolutely converge, we can compare them term by term and find
\begin{align*}
&\|\gamma_\Lambda^t(f)-\gamma_{\Lambda'}^t(f)\|_\infty\\
&\leq\Big\|\sum_{n=0}^\infty\int_0^tds_n\cdots\int_0^{s_2}ds_1\Bigg(\{\cdots\{\{f,V_{\Lambda,s_1}\},V_{\Lambda,s_2}\}\cdots,V_{\Lambda,s_n})\}\\
&\qquad-\{\cdots\{\{f,V_{\Lambda',s_1}\},V_{\Lambda',s_2}\}\cdots,V_{\Lambda',s_n}\}\Bigg)\Big\|_\infty\\
&=\Big\|\bigg(\sum_{(n,(\boldsymbol{k},\boldsymbol{l}))\in \coprod_{n=0}^\infty\Lambda^{2n}}-\sum_{(n,(\boldsymbol{k},\boldsymbol{l}))\in \coprod_{n=0}^\infty(\Lambda')^{2n}}\bigg)(\int_0^tds_n\cdots\int_0^{s_2}ds_1\\
&\quad\{\cdots\{\{f,V_{k_1l_1,s_1}\},V_{k_2l_2,s_2}\}\cdots,V_{k_nl_n,s_n}\})\Big\|_\infty\\
&\leq \sum_{(n,(\boldsymbol{k},\boldsymbol{l}))\in (\coprod_{n=0}^\infty\Gamma^{2n})\setminus F}\int_0^tds_n\cdots\int_0^{s_2}ds_1\\
&\quad\supnorm{\{\cdots\{\{f,V_{k_1l_1,s_1}\},V_{k_2l_2,s_2}\}\cdots,V_{k_nl_n,s_n}\}}\\
&<\epsilon,
\end{align*}
where in the second to last inequality we have used the fact that, because $\Lambda_\epsilon\subseteq\Lambda'$, we obtain $F\subseteq\coprod_{n=0}^\infty\Lambda_\epsilon^{2n}\subseteq\coprod_{n=0}^\infty(\Lambda')^{2n}$, and therefore $\coprod_{n=0}^\infty\Lambda^{2n}\setminus \coprod_{n=0}^\infty(\Lambda')^{2n}\subseteq (\coprod_{n=0}^\infty\Gamma^{2n})\setminus F$.



In summary, for $f\in \pi_{\Lambda_0}^*\DhR(\Omega_{\Lambda_0})$ and $|t|<t_\lambda$, the net $\{\gamma_\Lambda^t(f)\}_{\Lambda\Subset\Gamma}$ is Cauchy as $\Lambda_0\subseteq\Lambda\nearrow\Gamma$. Since
$$\Phi_{H_\Lambda^0}^t\circ\pi_{\Lambda_0}=\Phi_{H_{\Lambda_0}^0}^t\circ\pi_{\Lambda_0},$$
for $\Lambda\supseteq\Lambda_0$, it follows that
\begin{align*}
    \alpha_\Lambda^t(f)=f\circ\Phi^t_{H_\Lambda}\circ\pi_\Lambda=\gamma_\Lambda^t(f\circ\Phi^t_{H_\Lambda^0}\circ\pi_{\Lambda_0})=\gamma_\Lambda^t(f\circ\Phi_{H_{\Lambda_0}^0}^t\circ\pi_{\Lambda_0}).
\end{align*}
Note that $f\circ\Phi_{H_{\Lambda_0}^0}^t\circ\pi_{\Lambda_0}\in\pi_{\Lambda_0}^*\DhR(\Omega_{\Lambda_0})$ is a function independent of $\Lambda$.
Hence, if $\Lambda,\Lambda'$ grow large, 
\begin{align}\label{eq:Cauchy alpha Lambda}
    \|\alpha_\Lambda^t(f)-\alpha_{\Lambda'}^t(f)\|_\infty\to 0,
\end{align}
whenever $|t|<t_\lambda$.
\end{proof}

We now prove the main theorem of this section.

\begin{proof}[Proof of Theorem \ref{thm: main}]
The above proposition allows us to define
$$\alpha^t(f):=\lim_{\Lambda\nearrow\Gamma}\alpha^t_\Lambda(f)$$
for all $f\in\pi_{\Lambda_0}^*\DhR(\Omega_{\Lambda_0})$ and $|t|<t_\lambda$.
Let $\Lambda_0\subseteq\Lambda'\subseteq\Lambda\Subset\Gamma$. By the group properties of $\alpha_\Lambda^t$ and $\alpha_{\Lambda'}^t$,
\begin{align}\label{eq:t+t'}
\|\alpha_\Lambda^{t+t'}(f)-\alpha_{\Lambda'}^{t+t'}(f)\|_\infty\leq& \|\alpha_\Lambda^t(\alpha_\Lambda^{t'}(f)-\alpha_{\Lambda'}^{t'}(f))\|_\infty\nonumber\\
&+\|\alpha_{\Lambda}^t(\alpha_{\Lambda'}^{t'}(f))-\alpha_{\Lambda'}^t(\alpha_{\Lambda'}^{t'}(f))\|_\infty,
\end{align}
for all $|t|,|t'|<t_\lambda$. The first term is bounded by $\|\alpha_\Lambda^{t'}(f)-\alpha_{\Lambda'}^{t'}(f)\|_\infty$ as $\alpha^t_\Lambda$ is isometric. The second term can be estimated in the following manner. On account of  \eqref{eq:Cauchy alpha Lambda}, for any $\epsilon>0$ we may find $\Lambda''\Subset\Gamma$ sufficiently large, such that $\|\alpha_{\Lambda'}^{t'}(f)-\alpha_{\Lambda''}^{t'}(f)\|_\infty< \epsilon$ for all $\Lambda'\supseteq\Lambda''$.  Taking this $\Lambda''$, and observing that $\alpha_{\Lambda''}^{t'}(f)\in\pi_{\Lambda''}^*\CR(\Omega_{\Lambda''})$ (Theorem \ref{thm: main}), and that $\DhR(\Omega_{\Lambda''})\subset\CR(\Omega_{\Lambda''})$ is dense, we may find $g\in\pi_{\Lambda''}^*\DhR(\Omega_{\Lambda''})$ such that $\|g-\alpha_{\Lambda''}^{t'}(f)\|_\infty<\epsilon$. By the triangle inequality the previous implies for all $\Lambda'\supset\Lambda''$ that
\begin{align}\label{ineq: Dhoed}
    \|g-\alpha_{\Lambda'}^{t'}(f)\|_\infty\leq \|\alpha_{\Lambda'}^{t'}(f)-\alpha_{\Lambda''}^{t'}(f)\|_\infty+\|\alpha_{\Lambda''}^{t'}(f)-g\|_\infty\leq 2\epsilon.
\end{align}
Therefore, for $\Lambda''\subseteq\Lambda'\subseteq\Lambda$ the second term of \eqref{eq:t+t'} can be bounded as follows:
\begin{align*}
    &\|\alpha_{\Lambda}^t(\alpha_{\Lambda'}^{t'}(f))-\alpha_{\Lambda'}^t(\alpha_{\Lambda'}^{t'}(f))\|_\infty\leq
    \|\alpha_{\Lambda}^t(\alpha_{\Lambda'}^{t'}(f)-g)\|_\infty+ \\&\|\alpha_{\Lambda}^t(g)-\alpha_{\Lambda'}^t(g)\|_\infty+
    \|\alpha_{\Lambda'}^t(g-\alpha_{\Lambda'}^{t'}(f))\|_\infty\leq\\&
    2\|\alpha_{\Lambda'}^{t'}(f)-g\|_\infty+ \|\alpha_{\Lambda}^t(g)-\alpha_{\Lambda'}^t(g)\|_\infty\leq 4\epsilon+\|\alpha_{\Lambda}^t(g)-\alpha_{\Lambda'}^t(g)\|_\infty.
\end{align*}
Combining the above inequality with \eqref{eq:t+t'} yields
\begin{align*}
    \|\alpha_\Lambda^{t+t'}(f)-\alpha_{\Lambda'}^{t+t'}(f)\|_\infty\leq&  \|\alpha_\Lambda^{t'}(f)-\alpha_{\Lambda'}^{t'}(f)\|_\infty+4\epsilon+\|\alpha_{\Lambda}^t(g)-\alpha_{\Lambda'}^t(g)\|_\infty.
\end{align*}
On account of \eqref{eq:Cauchy alpha Lambda} we finally conclude
\begin{align*}
    \lim_{\Lambda,\Lambda'\nearrow\Gamma}\|\alpha_\Lambda^{t+t'}(f)-\alpha_{\Lambda'}^{t+t'}(f)\|_\infty=0\qquad (|t|,|t'|<t_\lambda).
\end{align*}
%
%
Repeating this argument yields
\begin{align}\label{limitexists}
    \lim_{\Lambda,\Lambda'\nearrow\Gamma}\|\alpha_\Lambda^t(f)-\alpha_{\Lambda'}^t(f)\|_\infty= 0\qquad\text{for all $t\in\R.$}
\end{align}
Hence, for each $f\in\pi_{\Lambda_0}^*\DhR(\Omega_{\Lambda_0})$ and each $t\in\R$, the net $\{\alpha_{\Lambda}^t(f)\}_{\Lambda\supseteq\Lambda_0}\subseteq\CR(\Omega)$ is a Cauchy net, and therefore has a $\|\cdot\|_\infty$-limit, which we denote by $\alpha^t(f)$.

We notice that \eqref{limitexists} defines $\alpha^t$ on a dense subspace of $\CR(\Omega)$, but in fact, $\alpha^t$ is isometric as seen from the computation
$$\supnorm{\alpha^t_{}(f)}=\lim_{\Lambda\nearrow\Gamma}\supnorm{\alpha^t_{\Lambda}(f)}=\lim_{\Lambda\nearrow\Gamma}\supnorm{f}=\supnorm{f}.$$
Therefore $\alpha^t_{}$ extends by continuity to an isometry
$$\alpha^t:\CR(\Omega)\to\CR(\Omega),$$
for all $t\in\R$. By density of $\pi_{\Lambda_0}^*\DhR(\Omega_{\Lambda_0})$ in $\pi_{\Lambda_0}^*\CR(\Omega_{\Lambda_0})$ we now conclude that the first statement of the theorem (i.e., \eqref{eq:def alpha^t}) holds.

Let us now deduce that the map $\CR(\Omega)\ni f\mapsto \alpha_{}^t(f)$ is a $*$-homomorphism. It suffices to prove this on $\pi_{\Lambda_0}^*\CR(\Omega_{\Lambda_0})$ by density, cf. \eqref{eq:CR closure of union}. Indeed, the homomorphism property follows from
\begin{align*}
||\alpha^t(f)\alpha^t(g)-\alpha^t(fg)||_\infty&= \lim_{\Lambda\nearrow\Gamma}||\alpha^t(f)\alpha^t(g)-\alpha_{\Lambda}^t(fg)||_\infty\\ 
&\leq\lim_{\Lambda\nearrow\Gamma}||\alpha^t(f)-\alpha_{\Lambda}^t(f)||_\infty||\alpha^t(g)||_\infty\\
&\quad+\lim_{\Lambda\nearrow\Gamma}||\alpha_{\Lambda}^t(f)||_\infty||\alpha^t(g)-\alpha_{\Lambda}^t(g)||_\infty\nonumber\\&=0,\qquad(f\in\pi_{\Lambda_0}^*\CR(\Omega_{\Lambda_0}))
\end{align*}
where we have used \eqref{eq:def alpha_Lambda} and the isometricity of $\alpha_{\Lambda}^t$. A similar argument shows that $f\mapsto \alpha_{}^t(f)$ preserves the adjoint, as well as linear combinations. To show that the map $t\mapsto \alpha^t_{}$ preserves the group structure we take $f\in \pi_{\Lambda_0}^*\CR(\Omega_{\Lambda_0})$, $t,t'\in\R$, and $\epsilon>0$ arbitrary. By \eqref{eq:CR closure of union} again, we find $\Lambda'\Subset\Gamma$ and $g\in\pi_{\Lambda'}^*\CR(\Omega_{\Lambda'})$ such that $\|g-\alpha^{t'}(f)\|_\infty<\epsilon$. Taking a limit over $\Lambda\supseteq\Lambda'$ and using that $t\mapsto\alpha^t_\Lambda$ is a group homomorphism, we find
  \begin{align*}
      &||\alpha_{}^{t+t'}(f)-\alpha_{}^{t}(\alpha_{}^{t'}(f))||_\infty\\
      &=\lim_{\Lambda\nearrow\Gamma}||\alpha_{\Lambda}^{t+t'}(f)-\alpha^{t}(\alpha_{}^{t'}(f))||_\infty\nonumber\\
      &\leq\lim_{\Lambda\nearrow\Gamma}||\alpha_{\Lambda}^{t}(\alpha_\Lambda^{t'}(f))-\alpha^{t}(g)||_\infty+\lim_{\Lambda\nearrow\Gamma}\|\alpha^{t}(g)-\alpha^{t}(\alpha^{t'}(f))||_\infty\nonumber\\
      &<\lim_{\Lambda\nearrow\Gamma}||\alpha_{\Lambda}^{t}(\alpha_\Lambda^{t'}(f))-\alpha_\Lambda^{t}(g)||_\infty+\epsilon\\
      &<2\epsilon,
 \end{align*}
where we have used isometricity of $\alpha^t,\alpha_\Lambda^t$ and  \eqref{limitexists}.  
As $\epsilon>0$ was arbitrary, and by \eqref{eq:CR closure of union}, it follows that
$$\alpha_{}^{t+t'}(f)=\alpha_{}^{t}(\alpha_{}^{t'}(f))\qquad(t,t'\in\R,f\in\CR(\Omega)),$$
hence $t\mapsto\alpha_{}^t$ is a one-parameter subgroup of automorphisms of $\CR(\Omega)$, concluding the proof of the theorem.
\end{proof}

\section{Strong continuity}\label{Sec:strong continuity}
In Theorem \ref{thm: CR stable} we have established the existence of the dynamics in the thermodynamic limit.
This however does not guarantee the desired notion of {\em strong continuity}, also called {\em pointwise norm continuity}, used in C*-dynamical systems, i.e. convergence of $\lim_{t\to0}\supnorm{\alpha^t_{}(f)-f}$. 
In this final part we investigate this form of continuity, taking inspiration from \cite[Proposition 7.2]{Buchholz_Grundling_2008} and the subsequent discussion. We firstly introduce the subalgebra
\begin{align}
    \mathfrak{K}:=C^*(\mathfrak{K}_\Lambda \ | \ \Lambda\Subset\Gamma)+ \mathbb{C}1\subset \CR(\Omega),
\end{align}
where $\mathfrak{K}_\Lambda=\{g\circ \pi_\Lambda\ |\ g\in C_0(\Omega_\Lambda) \}$. 
It is not difficult to see that $\mathfrak{K}$ is a proper non-trivial subalgebra of $\CR(\Omega)$. 
\begin{lemma}
Consider the situation of Theorem \ref{thm: main}. Then the automorphic action $\alpha:\R\to \textnormal{Aut}(\CR(\Omega))$ is strongly continuous on $\mathfrak{K}$.
\end{lemma}
\begin{proof}
     It suffices to show $\lim_{t\to0}\|\alpha^t(f)-f\|=0$ and by standard arguments we may furthermore assume that $f=g\circ \pi_\Lambda$ for a certain $g\in C_0(\Omega_\Lambda)$ and a fixed $\Lambda\Subset\Gamma$. We shall first prove
     \begin{align}\label{finite case compact}
         \lim_{t\to 0}||\alpha_{\Lambda}^t(g\circ \pi_\Lambda)-g\circ \pi_\Lambda ||_\infty=0.
     \end{align}
     To do so, we 
     note that
    \begin{align}\label{identity strongly cts}
         \|\alpha_\Lambda^t(g\circ\pi_\Lambda)-g\circ\pi_\Lambda\|_\infty&=\|g\circ\Phi^t_{H_\Lambda}\circ\pi_\Lambda-g\circ\pi_\Lambda\|_\infty\nonumber\\
         &=\|g\circ\Phi^t_{H_\Lambda}-g\|_\infty.
    \end{align}
     We now observe that $(C_0(\Omega_\Lambda), (\Phi^t_{H_\Lambda})^*)$ is a C*-dynamical system. Indeed, our assumptions on the potential ensure that the Hamiltonian flow is complete and since $(\Phi^t_{H_\Lambda})^*$ is a homeomorphism it preserves compacts, and hence leaves $C_0(\Omega_\Lambda)$ invariant. Strong continuity then follows from standard results (e.g. \cite{Moretti_vandeVen_2021}). By applying this observation to \eqref{identity strongly cts} and taking the limit $t\to 0$ we conclude that \eqref{finite case compact} holds.

     On account of \eqref{eq:def alpha^t}, given $\epsilon>0$, we find $\Lambda'\Subset\Gamma$ such that $\Lambda\subseteq\Lambda'$ and
     $$||\alpha_{}^t(g\circ \pi_{\Lambda}) - \alpha_{\Lambda'}^t(g\circ \pi_\Lambda)||_\infty\leq \epsilon/3.$$
     
     In order to compare the dynamics of $H_{\Lambda'}$ with that of $H_\Lambda$, we use \eqref{ham 3} and apply Gr\"{o}nwall's inequality  (cf. Lemma \ref{lem:Gronwall's inequality}) with $U=\Omega_{\Lambda'}$, $F=X_{H_{\Lambda}+H_{\Lambda'-\Lambda}^0}$ (the second term corresponds to $|\Lambda'|-|\Lambda|$ simple harmonic oscillators), $G=X_{H_{\Lambda'}}$, $t_0=0$, $t_1=t$, $y_0=z_0=\omega$,
     $$c:=\sup_{\omega'\in\Omega_{\Lambda'}}||G(\omega')-F(\omega')||\leq ||\nabla (V_{\Lambda'}-V_{\Lambda})||_\infty<\infty,$$
     $\varphi:=c$ and obtain
     \begin{align}\label{inequality fundamental}
     \sup_{\omega\in\Omega_{\Lambda'}}\|\Phi^t_{F}(\omega) - \Phi_{G}^t(\omega)\|\leq cte^{tC},
     \end{align}
     where $C$ denotes the Lipschitz constant of $G$.
     
     Because $g\circ \pi_\Lambda$ is uniformly continuous, the bound \eqref{inequality fundamental} implies that
     $$\sup_{\omega\in\Omega_{\Lambda'}}|g(\pi_\Lambda(\Phi^t_{H_{\Lambda}+H_{\Lambda'-\Lambda}^0}(\omega))) - g(\pi_\Lambda(\Phi_{H_{\Lambda'}}^t(\omega)))|\to 0$$
     as $t\to0$. We note that $\pi_\Lambda(\Phi^t_{H_{\Lambda}+H_{\Lambda'-\Lambda}^0}(\omega))=\pi_\Lambda(\Phi^t_{H_\Lambda}(\omega))$, and conclude
     \begin{align}\label{eq:strong continuity from Gronwall}
     \lim_{t\to0}\supnorm{\alpha_\Lambda^t(g\circ \pi_\Lambda)-\alpha^t_{\Lambda'}(g\circ \pi_\Lambda)}=0.
     \end{align}
     We combine our above results by the estimates
     \begin{align*}
         \nonumber&||\alpha_{}^t(g\circ \pi_\Lambda) - g\circ \pi_\Lambda||_\infty
         \\
         &\leq ||\alpha_{}^t(g\circ \pi_\Lambda) - \alpha_{\Lambda'}^t(g\circ \pi_\Lambda)||_\infty
         + || \alpha_{\Lambda'}^t(g\circ \pi_\Lambda)- \alpha_\Lambda^t(g\circ \pi_\Lambda)||_\infty\nonumber\\
         &\quad+
         ||\alpha_{\Lambda}^t(g\circ \pi_\Lambda)- g\circ \pi_\Lambda ||_\infty  \nonumber\\
         & \leq  \epsilon/3 + || \alpha_{\Lambda'}^t(g\circ \pi_\Lambda)- \alpha_\Lambda^t(g\circ \pi_\Lambda)||_\infty+
         ||\alpha_{\Lambda}^t(g\circ \pi_\Lambda)- g\circ \pi_\Lambda ||_\infty.\nonumber
     \end{align*}
    The second and third term can each be made smaller than $\epsilon/3$ for small enough $t$, by using \eqref{finite case compact} and \eqref{eq:strong continuity from Gronwall}. Strong continuity follows.
\end{proof}

Because of the interactions defined in terms of $V_{kl}$, the algebra $\mathfrak{K}$ is in general not stable under the dynamics. We therefore define the C*-algebra 
\begin{align}\label{eq:def L}
\mathfrak{L}=C^*(\cup_{t\in\R}\alpha_{}^t(\mathfrak{K}))
\end{align}
generated by $\alpha_{}^t(\mathfrak{K})$, ($t\in\mathbb{R}$), which is by definition stable under the action of the map $t\mapsto\alpha_{}^t$. It may be clear that the action  $\alpha_{}:\R\to \text{Aut}(\mathfrak{L})$  is still strongly continuous on $\mathfrak{L}$, where $\text{Aut}(\mathfrak{L})$ is the group of *-automorphisms of $\mathfrak{L}$. In summary, we have obtained our main theorem.
\begin{theorem}\label{thm: final goal}
    For any countable set $\Gamma$, constants $m_k,\nu_k>0$, and functions $V_{kl}:\R^d\to\R$ for all $k,l\in\Gamma$ satisfying Assumption \ref{conditions}, the C*-subalgebra $\mathfrak{L}\subset\CR(\Omega)$ defined by \eqref{eq:def L}, and the dynamics $\alpha^\R|_\mathfrak{L}$ obtained by restricting the time-evolution $\alpha^\R$ of Theorem \ref{thm: CR stable}, the tuple 
    $$(\mathfrak{L},\alpha^\R|_\mathfrak{L})$$
    is a C*-dynamical system. I.e., the map $\R\to\textnormal{Aut}(\mathfrak{L}),t\mapsto\alpha^t|_{\mathfrak{L}}$ is a well-defined strongly continuous group homomorphism.
\end{theorem}

The C*-dynamical systems $(\mathfrak L,\alpha^\R|_{\mathfrak L})$ themselves will depend on the specific system, i.e., on $\{m_k\}_{k\in\Gamma},\{\nu_k\}_{k\in\Gamma}$ and $\{V_{kl}\}_{k,l\in\Gamma}$. It is however nice that there is an overarching dynamical C*-algebra $(\CR(\Omega),\alpha^\R)$ which is sufficiently easy -- namely given by an inductive limit and on each finite subsystem similarly well-behaved as the $C_0$-functions -- in order to simplify subsequent analysis.

\begin{remark}
By Theorem \ref{thm: final goal} there also exists a nontrivial maximal subalgebra of $\CR(\Omega)$ on which the homomorphism  $t\mapsto \alpha^t$ is strongly continuous, but this algebra might be harder to describe. For a system of two non interacting particles, for instance, this maximal algebra is strictly larger than $\mathfrak L$. To see this, let the particles $k,l\in\Gamma$ have unit mass and force constant, and consider the projection $P(p_1,q_1,p_2,q_2):=(p_1-p_2,q_1-q_2)$. If one now takes $g(x,y)=e^{-\|x\|^2-\|y\|^2}$, it follows that $f:=g\circ P:(p_1,q_1,p_2,q_2)\mapsto e^{-(p_1-p_2)^2-(q_1-q_2)^2}$, and  the function $t\mapsto\alpha^t(f)$ is continuous, since it is constant. However, $f$ is not in $\mathfrak K=(C_0\oplus\mathbb C1)\otimes(C_0\oplus\mathbb C1)$, and, since $\alpha^t=\alpha^t_{\{k\}}\otimes\alpha^t_{\{l\}}$, $f$ is not in $\mathfrak L_\Gamma=\mathfrak K$.
\end{remark}

\section{Discussion}\label{sct:Discussion}
Similar to how the non-commutative resolvent algebra can model quantum statistical mechanics, the intriguing mathematical structure of the commutative resolvent algebra allows to study classical statistical mechanics. It is a C*-algebra possessing  many desirable features for the treatment of finite and infinite dimensional classical systems.  Indeed, due to its non-trivial ideal structure, the commutative resolvent algebra may serve as a container for C*-dynamical systems for a large class of classical Hamiltonians and can therefore be utilized to investigate emergent phenomena that arise in the thermodynamic limit.  Of particular importance in this context is understanding the  algebraic state space and investigating the notion of the classical Kubo--Martin--Schwinger (KMS) condition, which analogous to the renown Dobrushin--Lanford--Ruelle (DLR) equations, is believed to encode the thermodynamic properties of the underlying classical system \cite{Gallavotti_Verboven_1975}. This may lead to the possible equivalence between the DLR equations and the classical KMS condition, which has been proved in the continuum in \cite{Aizenman_Goldstein_Gruber_Lebowitz_Martin_1977}, in the context of classical spin lattice systems in \cite{Drago_vandeVen_2023}, and in the classical limit of mean-field quantum spin systems this issue has been investigated in \cite{vandeVen_2022_2}.
It is worthwhile to say that the relationship between states and probability measures is no longer necessarily described by the celebrated Riesz-Markov-Kakutani representation theorem, as in the standard  algebraic set-up where one considers $C_0$-functions.

We saw that the proof of our first main result could be successfully split into a finite-dimensional argument and an argument involving the thermodynamic limit.
Especially with respect to the thermodynamic limit, we have seen that techniques developed for quantum mechanics -- in particular, the well-known Dyson series approach of Robinson -- can be carried through to the classical case when one is very careful about adding sufficient smoothness assumptions on the potentials. The viewpoint that emerges from this technical procedure is that classical dynamics is harder to understand algebraically than quantum dynamics, although there are useful parallels. Indeed, we expect our result to have an analogue in the quantum case with significantly milder assumptions on the potentials.
This strategy of proving quantum results after establishing the corresponding classical result has proven successful in similar situations, namely in \cite{vanNuland_Stienstra_2019}, and in the passage from \cite{vanNuland_2019} to \cite{Buchholz-vanNuland}.



Finally, it cannot go unmentioned that the inclusion of Coulomb interactions -- instead of merely bounded interactions $V_{kl}$ -- into our model would be a valuable addition, and forms a great challenging question at present.

If nothing else, our result suggests that further study and applications of the commutative resolvent algebra to classical particle systems are definitely worthwhile. 

\section{Acknowledgements}
The authors thank Pieter Naaijkens,  Nicol\`{o} Drago, and Valter Moretti for their useful feedback. TvN was supported by ARC grant FL17010005. CvdV was supported by a postdoctoral fellowship granted by the Alexander von Humboldt Foundation (Germany). 

\paragraph{Data availability statement.} Data sharing is not applicable to this article as no new data were created or analysed in this study.

\paragraph{Compliance with Ethical Standards and conflict of interest statement.} The authors certify that they have no affiliations with or involvement in any organization or entity with any financial interest or non-financial interest in the subject matter discussed in this manuscript
The research involved no participation of humans or other animals except the authors themselves, and thus complies with the standards of ethics and informed consent of the journal.

\end{document}